\documentclass[12pt, a4paper]{article}

\usepackage[utf8]{inputenc}      
\usepackage[T1]{fontenc}         
\usepackage[english]{babel}      
\usepackage{amsmath, amsthm, amssymb, amsfonts, mathrsfs} 
\usepackage{geometry}            
\usepackage{enumerate}           
\usepackage{hyperref}            
\usepackage{cite}                
\usepackage{comment}
\usepackage{xcolor}
\usepackage[normalem]{ulem}
\usepackage{cleveref}
\numberwithin{equation}{section}

\newtheorem{theorem}{Theorem}[section]     
\newtheorem{lemma}[theorem]{Lemma}         
\newtheorem{proposition}[theorem]{Proposition} 

\theoremstyle{definition}                  
\newtheorem{definition}[theorem]{Definition}

\newtheorem{remark}[theorem]{Remark}

\newcommand{\Prob}{\mathbb{P}}
\newcommand{\abs}[1]{\left|#1\right|}
\newcommand{\E}{\mathbb{E}}
\newcommand{\norm}[1]{\left\|#1\right\|}
\newcommand{\inner}[2]{\left\langle #1, #2 \right\rangle}
\newcommand{\cH}{\mathcal{H}}
\title{An Inverse Random Source Problem for the Moore-Gibson-Thompson Equation Driven by Fractional Brownian Motion}
\author{Lingxi Gao, Yeqiong Ye and Ting Zhou \\ \small Department of Mathematics \\ \small Zhejiang University}
\date{}

\begin{document}

\maketitle

\begin{abstract}
    \noindent
    In this paper, we consider an inverse random source problem for the stochastic Moore-Gibson-Thompson equation driven by fractional Brownian motion with Hurst index $H \in (0, 1)$ of the form $f_1(x)g_1(t)\dot{B}^H(t)+f_2(x)g_2(t)$. Given the random source, existence and uniqueness of mild solutions are verified. For the inverse problem, the uniqueness of recovering the strength $f_i(x)$ if the time functions $g_i$ are known and $g_i(t)$ if the spatial functions $f_i$ are known when $H \in(0,1)$ from the boundary flux on a special nonempty open subset is proved.
\end{abstract}

\section{Introduction}
    \subsection{Statement of the Problem}
The Moore-Gibson-Thompson (MGT) equation is a third-order-in-time partial differential equation that arises in the mathematical modeling of acoustic wave propagation in viscous thermally relaxing fluids. It was originally derived from the works of Moore, Gibson, and Thompson\cite{doi:10.2514/8.8418, thompson1972compressible} as a refinement of classical acoustic models, and can be formally written as
\begin{equation*}
    \partial_{ttt} u(x, t)+\alpha \partial_{tt} u(x, t)-b\Delta \partial_t u(x, t) -c^2 \Delta u(x, t)=F(x, t) .
\end{equation*}
It has attracted considerable attention in both the mathematical and engineering communities due to its rich mathematical structure and wide applicability in fields such as medical ultrasound, nondestructive testing, and sonar technology.

Despite the growing literature on deterministic MGT equations, the stochastic counterpart — where the source term is driven by random noise — has received comparatively little attention. Inverse random source problems have been investigated for several stochastic evolution equations, including stochastic wave equations, stochastic time-fractional diffusion and diffusion-wave equations, and stochastic biharmonic wave equations. These works usually recover statistical information of the source from expectation, covariance, final-time observations, boundary measurements, or far-field data. \cite{XiaoliFeng2022} determined the statistical properties of the source for stochastic wave equation from the expectation and covariance of the final-time data. \cite{Lassas_2023} proved the uniqueness of the random source under the boundary measurements for stochastic time-fractional diffusion wave equations. \cite{doi:10.1137/21M1429138} considered the stochastic biharmonic wave equation with a microlocally isotropic Gaussian random source and showed that the strength of the random source, appearing in the principal symbol of the covariance operator, can be uniquely determined by a single realization of the magnitude of the wave field averaged over a frequency band, with probability one. And recently, \cite{https://doi.org/10.1002/mma.70574} proved uniqueness of recovering two initial values and a random source simultaneously by final observations via phase retrieval for stochastic damped wave equations.

Motivated by these developments, this paper studies an inverse random source problem for the stochastic Moore-Gibson-Thompson equation  driven by Gaussian random noise, formulated as follows: 
\begin{flalign}
    \begin{cases} 
        \partial_{ttt} u(x, t)+\alpha \partial_{tt} u(x, t)-b\Delta \partial_t u(x, t) -c^2 \Delta u(x, t)=F(x, t) & (x, t) \in D \times [0, T],\\
        u(x, t)=0 & (x, t) \in \partial D \times [0, T],\\
        u(x, 0)=\partial_t u(x, 0)=\partial_{tt} u(x, 0)=0 & x \in D.  \label{MGT}
    \end{cases}
\end{flalign}
In contrast to the above works, the present model is a third-order-in-time acoustic equation with memory-type damping structure. 
Here, the domain \(D\) is assumed to be bounded and open with $C^2$ boundary. 
In many applications of acoustic wave propagation, the effective acoustic source cannot be regarded as perfectly deterministic. Variations in the operating conditions of the source, environmental disturbances, imperfect coupling between the source and the surrounding medium, and other unresolved effects may introduce random fluctuations into the excitation. Such effects are particularly relevant in thermally relaxing and viscous media, where the MGT equation provides a more appropriate description. To account simultaneously for the nominal excitation and its random fluctuations, we consider a source of the form
\begin{equation*}
    F(x, t)=f_1(x)g_1(t)\dot{B}^H(t)+f_2(x)g_2(t),
\end{equation*}
where the second term represents a deterministic source component, while the first describes a randomly fluctuating component. The use of fractional Brownian motion $B^H$ allows temporal correlations in the random excitation to be incorporated into the model, with the classical white-noise case recovered when (H=1/2).

The corresponding inverse problem can be interpreted as identifying an unknown acoustic source from measurements collected on an accessible portion of the boundary. 

In our setting, we assume that the parameters \( \alpha,\ b,\ c^2\) are constant and strictly positive, and the parameter $\gamma := \alpha-\frac{c^2}{b}>0$ plays an important role in the well-posedness of the direct problem. Throughout this paper, we assume $\gamma>0$, under which the associated solution semigroup is uniformly exponentially stable.\cite{article} 

In practice, the boundary flux $\partial_\nu u(x, t, \omega)$ recorded in a single experiment contains contributions from both the deterministic source components $f_i, g_i$ and the stochastic noise. Since these two contributions are entangled in each individual realization, it is in general impossible to separate them from a single measurement. Hence in stochastic inverse source problems, one possible approach is to use statistical quantities of the boundary measurement, such as its expectation and variance, as the observation data. The expectation depends linearly on the deterministic source component and therefore allows the usual difference argument. By contrast, the variance depends quadratically on the amplitude of the stochastic source. Consequently, if two stochastic sources produce the same variance of the boundary flux, the difference of the two variances cannot in general be represented as the variance generated by the difference of the sources. In particular,
\begin{equation*}
    \operatorname{Var}[\partial_{\nu}u_F]=\operatorname{Var}[\partial_{\nu}u_{\tilde{F}}],
\end{equation*}
does not imply 
\begin{equation*}
    \operatorname{Var}[\partial_{\nu}u_{F-\tilde{F}}]=0.
\end{equation*}
Therefore, an argument based only on the variance typically yields a zero-observation uniqueness result, namely that vanishing statistical data imply the vanishing of the corresponding source, but it does not directly establish uniqueness between two arbitrary nonzero sources. To obtain genuine uniqueness, we take the random boundary flux itself as the observation data. More precisely, we consider the observed data as
\begin{equation*}
    \partial_{\nu}u|_{\Gamma\times (0, T)} \in L^2\left(\Omega; L^2\left(\Gamma\times (0, T)\right)\right).
\end{equation*}
When two solutions are compared, they are assumed to be defined on the same probability space and driven by the same fractional Brownian motion. Since the forward problem is linear with respect to the source terms, equality of the random boundary measurements,
\begin{equation*}
    \partial_{\nu}u|_{\Gamma\times (0, T)}=\partial_{\nu}\tilde{u}|_{\Gamma\times (0, T)} \text{ in } L^2\left(\Omega; L^2\left(\Gamma\times (0, T)\right)\right)
\end{equation*}
implies that the difference $u-\tilde{u}$
 has vanishing boundary flux and is generated by the differences of the corresponding source profiles. Hence, a zero-observation injectivity result can be applied to the difference equation to establish uniqueness for two arbitrary admissible sources. Expectations and second moments are still used in the proof as analytical tools for separating the deterministic and stochastic source components, but they are not regarded as the measurement data.

\subsection{Main Results} 
We first establish the well-posedness and regularity of the mild solution to lay the groundwork for the inverse analysis.
\begin{theorem} \label{thm_existence}
         Assume $f_1 \in  H^1_0(D)$, $f_2 \in L^2(D)$ and $g_2 \in L^2(0, T)$. Additionally assume that $g_1 \in C^{0, \rho}([0, T])$ with $\rho > \frac{1}{2}-H$ for $H \in \left(0, \frac{1}{2}\right)$, and $g_1 \in L^2(0, T)$ for $H \in \left[\frac{1}{2}, 1\right)$, then the stochastic MGT equation \eqref{evolution_equation} admits a unique mild solution given by
    \begin{equation}\label{eq:mild-solution}
        X(t) = \int_0^t e^{A(t-s)} \Phi(s) dB^H(s)+\int_0^t e^{A(t-s)} \tilde{\Phi}(s) ds.
    \end{equation}
    Moreover, $X(t)\in L^2(\Omega;\mathcal{H})$ for all $t\in [0,T]$.
    \end{theorem}

Based on the well-posedness and the hidden regularity of the boundary flux, we establish the uniqueness of the inverse random source problem from the boundary measurement. The following two theorems state the unique recovery of the spatial profiles $f_i(x)$ and the temporal components $g_i(t)$, respectively.
\begin{theorem}\label{thm_unique-f}
    Under the assumptions of Theorem \ref{thm_existence}, assume further that $g_i \in C^1([0, T])$ satisfying $g_i(0) \ne 0$ for $i=1, \ 2$. If $\Gamma$ and $T$ satisfy the geometric condition \eqref{geometric} and time conditions
    \begin{equation}
        T > \frac{1}{\sqrt{b}}\sup_{x \in D} |x - x_0|. \label{time}
    \end{equation}
    Then, 
    \begin{flalign*}
        \partial_{\nu}u(x, t)=0 \in L^2 \left(\Omega;L^2\left(\Gamma \times (0, T)\right)\right)
    \end{flalign*}
    implies that
    \begin{flalign*}
        \mathbb{E}\left[ \partial_{\nu} u(x, t) \right]=\operatorname{Var} \left[ \partial_{\nu} u(x, t) \right]=0 \text{ in }L^2\left(\Gamma \times (0, T)\right).
    \end{flalign*}
    Therefore $f_1(x)=f_2(x)=0$ a.e. in $D$.
\end{theorem} 

\begin{theorem}\label{thm_unique-g}
    Under the assumptions of Theorem \ref{thm_existence}, assume further that $f_i \in H_0^1(D) \cap H^2(D)$, $\partial_\nu f_i(x) \ne 0$ in $L^2(\Gamma)$ for $i=1,2$ and $g_1 \in C^1([0, T])$. Then,
    \begin{flalign*}
        \partial_{\nu} u(x, t)=0 \in L^2 \left(\Omega;L^2\left(\Gamma \times (0, T)\right)\right)
    \end{flalign*}
    implies that $g_1(t) = g_2(t) = 0$ a.e. in $(0,T)$.
\end{theorem}
The rest of the paper is organized as follows. Section 2 introduces necessary preliminaries on fractional Brownian motion and SPDEs. Section 3 focuses on the direct problem, proving Theorem \ref{thm_existence} by establishing the well-posedness and the hidden regularity of the boundary flux for the stochastic MGT equation. Section 4 is devoted to the proof of Theorem \ref{thm_unique-f} regarding the unique recovery of the spatial intensity $f_i(x)$ and Theorem \ref{thm_unique-g} concerning the unique determination of the temporal source $g_i(t)$.

\section{Preliminaries}
\noindent
    We give a brief introduction to fractional Brownian motions in Hilbert spaces and Wiener integral with respect to the fractional Brownian motion. The details can be found in \cite{biagini2008stochastic, nualart2006malliavin, mishura2008stochastic, Tindel2003}, and \cite[Appendix]{Feng_2020}.

    Let $(U, \langle \cdot, \cdot \rangle)$ be a separable Hilbert space 
    endowed with the norm $\|\cdot\|$ induced by the inner product 
    $\langle \cdot, \cdot \rangle$. Let \(T >0 \) and $(\Omega, \mathcal{F}, \Prob)$ be a complete probability space, where $\Omega$ is the sample space, $\mathcal{F}$ is a $\sigma$-algebra on $\Omega$ and $\Prob$ is the probability measure on $(\Omega, \mathcal{F})$. In the sequel, the dependence of random variables on the sample $\omega \in \Omega$ will be omitted unless it is necessary to avoid confusion. We fix a time interval $[0, T]$. 
    \begin{definition}
        A one-dimensional centered Gaussian process $\{B^H(t)\}_{t \ge 0}$ is called fractional Brownian motion (fBm) with Hurst index $H \in (0, 1)$ if $B^H(0)=0$ and
        \begin{equation*}
            \operatorname{Cov}\left[ B^H(t),\ B^H(s)\right]=\frac{1}{2}\left(t^{2H}+s^{2H}-|t-s|^{2H}\right):=R_H(t, s).
        \end{equation*}
        In particular, if $H=\frac{1}{2}$, it reduces to the standard one-dimensional Wiener process. 
    \end{definition}
    
        It is worth noting that modern stochastic calculus provides various approaches to define stochastic integrals with respect to fBm for random integrands. These include the Wiener integral with respect to fractional Brownian motion, Skorohod integral based on Malliavin calculus and pathwise Riemann-Stieltjes integrals (see, e.g., \cite{biagini2008stochastic, nualart2006malliavin, mishura2008stochastic}). However, since the integrands arising in our specific inverse problem are strictly deterministic functions, we do not need to invoke these complex frameworks. Instead, it is sufficient and rigorous to understand the stochastic integration with respect to fBm in the sense of Wiener integrals. 
        
        Consider a fBm $\{B^H(t)\}_{t \in [0, T]}$ with Hurst parameter $H \in (0, 1)$. We denote by $\mathcal{E}$ the set of step functions on $[0, T]$. Let $\mathcal{H}_0$ be the Hilbert space defined as the closure of $\mathcal{E}$ with respect to the scalar product
        $$\langle \mathbf{1}_{[0,t]}, \mathbf{1}_{[0,s]} \rangle_{\mathcal{H}_0} = R_H(t, s).$$
        For $\varphi \in \mathcal{E}$ of the form $\varphi=\sum_{i=1}^{n}a_i\mathbf{1}_{\left(t_i, t_{i+1}\right]}(t)$, we define its Wiener integral with respect to the fractional Brownian motion as \cite{fractionalBrownian,Tindel2003}
        \begin{equation}\label{eq:defWiener}
            \int_0^T \varphi(s) dB^H(s) = \sum_{i=1}^n a_i \left( B^H(t_{i+1}) - B^H(t_{i})) \right).
        \end{equation}
    Obviously, the mapping
    \begin{equation*}
        \varphi = \sum_{i=1}^n a_i \mathbf{1}_{(t_i, t_{i+1}]} \rightarrow \int_0^T \varphi(s) dB^H(s) 
    \end{equation*}
    is an isometry between $\mathcal{E}$ and the linear space $\text{span}\{B^H(t), t \in [0,T]\}$ viewed as a subspace of $L^2(\Omega)$ and it can be extended to an isometry between $\mathcal{H}_0$ and the first Wiener chaos of the fractional Brownian motion $\overline{\operatorname{span}}^{L^2(\Omega)}\{B^H(t), t \in [0,T]\}$. The image of an element $\varphi \in \mathcal{H}_0$ under this isometry is called the Wiener integral of $\varphi$ with respect to $B^H$.

        For any $t\in [0,T]$ and deterministic $\psi,\ \varphi \in \mathcal{H}_0$, \\
        \noindent \textbf{case 1: $H=\frac{1}{2}$}
        \begin{flalign}\label{eq:H=1/2}
            &\mathbb{E}\left[ \int_0^t \psi(s) \; dB^H(s) \int_0^t \varphi(s) \; dB^H(s) \right]=\int_0^t \psi(s)\varphi(s) \; ds,
        \end{flalign}
        \noindent \textbf{case 2: $H \in (\frac{1}{2},1)$}
        \begin{flalign}\label{eq:H>1/2}
            &\mathbb{E}\left[ \int_0^t \psi(s) \; dB^H(s) \int_0^t \varphi(s) \; dB^H(s) \right]=H(2H-1)\int_0^t\int_0^t \psi(s)\varphi(r) |s-r|^{2H-2} \; ds\; dr,  
        \end{flalign}
        \noindent \textbf{case 3: $H \in (0,\frac{1}{2})$}, define $K_H(t, s)$ as
        \begin{flalign*}
            K_H(t, s):=c_H\left[\left(\frac{t}{s}\right)^{H-\frac{1}{2}}(t-s)^{H-\frac{1}{2}}-\left(H-\frac{1}{2}\right) s^{\frac{1}{2}-H} \int_{s}^{t} u^{H-\frac{3}{2}}(u-s)^{H-\frac{1}{2}} d u\right],
        \end{flalign*}
        where $c_H=\left(\frac{2H}{(1-2H)\beta\left(1-2H,\ H+1/2\right)}\right)^{\frac{1}{2}}$ and $\beta(\cdot, \cdot)$ denotes the beta function. 
        Define the linear operator $K^*_H(t)$ from the step functions on $(0, t)$ to $L^2(0, t)$ as 
        \begin{flalign}\label{eq:K_H^*}
            \left(K_H^*(t)\varphi\right)(s) :=K_H(t, s)\varphi(s)+\int_s^t \left(\varphi(\tau)-\varphi(s)\right)\frac{\partial}{\partial\tau}K_H(\tau, s)d\tau.
        \end{flalign}
        By the definitions of $K_H$ and $K_H^*(t)$, it can be derived that
        \begin{flalign}\label{eq:H<1/2}
            &\mathbb{E}\left[ \int_0^t \psi(s) \; dB^H(s) \int_0^t \varphi(s) \; dB^H(s) \right]=\int_{0}^{t} \left(K_H^*(t)\psi\right)(s)\cdot \left(K_H^*(t)\varphi\right)(s)  \nonumber\\
            =&\int_{0}^{t} \left\{ c_{H}\left[\left(\frac{t}{s}\right)^{H-\frac{1}{2}}(t-s)^{H-\frac{1}{2}}-\left(H-\frac{1}{2}\right) s^{\frac{1}{2}-H} \int_{s}^{t} u^{H-\frac{3}{2}}(u-s)^{H-\frac{1}{2}} d u\right] \psi(s) \right.\nonumber \\
            &\quad \left. + \int_{s}^{t}(\psi(u)-\psi(s)) c_{H}\left(\frac{u}{s}\right)^{H-\frac{1}{2}}(u-s)^{H-\frac{3}{2}} d u \right\} \nonumber\\
            &\quad \times \left\{ c_{H}\left[\left(\frac{t}{s}\right)^{H-\frac{1}{2}}(t-s)^{H-\frac{1}{2}}-\left(H-\frac{1}{2}\right) s^{\frac{1}{2}-H} \int_{s}^{t} u^{H-\frac{3}{2}}(u-s)^{H-\frac{1}{2}} d u\right] \varphi(s) \right. \nonumber\\
            &\quad \left. + \int_{s}^{t}(\varphi(u)-\varphi(s)) c_{H}\left(\frac{u}{s}\right)^{H-\frac{1}{2}}(u-s)^{H-\frac{3}{2}} d u \right\} d s,        
        \end{flalign}

        There are some useful properties of the integral kernel $K_H(t, s)$ and $\frac{\partial}{\partial t}K_H(t, s)$ (see \cite[Proposition 5.1.3]{nualart2006malliavin} and \cite[Theorem 3.2]{Decreusefond1999}):
        \begin{flalign}
        	R_H(t, s)=\int_{0}^{t \wedge s}K_H(t, u)K_H(s, u)\;du, \label{eq_kernel_1}
        \end{flalign}
        \begin{flalign}
        	\left|K_H(t, s)\right| \le c(H)\left(t-s\right)^{H-\frac{1}{2}}s^{H-\frac{1}{2}}. \label{eq_kernel_2}
        \end{flalign}
        Noticing that 
        \begin{flalign}
        	\frac{\partial}{\partial t}K_H(t, s)=c_H\left(H-\frac{1}{2}\right)\left(\frac{t}{s}\right)^{H-\frac{1}{2}}\left(t-s\right)^{H-\frac{3}{2}}, \label{eq_kernel_3}
        \end{flalign}
        the estimate of $\frac{\partial}{\partial t}K_H(t, s)$ holds (see \cite[Equation 5.46]{nualart2006malliavin})
        \begin{flalign}
        	\left|\frac{\partial}{\partial t}K_H(t, s)\right|\le c_H\left(\frac{1}{2}-H \right)\left(t-s\right)^{H-\frac{3}{2}}.\label{eq_kernel_4}
        \end{flalign}

        If $H<\frac{1}{2}$, $\mathcal{H}_0$ is a class of functions that contains $C^{0, \gamma}\left(\left[0, T\right]\right)$ if $\gamma > \frac{1}{2}-H$. If $H=\frac{1}{2}$, then $\mathcal{H}_0=L^2\left(\left[0, T\right]\right)$, and if $H>\frac{1}{2}$, $\mathcal{H}_0$ contains the space of functions that the value of \eqref{eq:H>1/2} is strictly less than infinity.
        
     Since the trajectories of $B^H(t)$ are not differentiable, the stochastic MGT equation in \eqref{MGT} does not hold pointwise; it should be interpreted as an integral equation. We introduce the strong solution and mild solution of the stochastic equation. Consider the linear stochastic equation with additive noise on \([0, T]\) in the Hilbert space $\mathcal{H}$,
    \begin{flalign}
        \begin{cases}
            dX(t)=AX(t)+\Phi(t)\;dB^H(t)+\tilde{\Phi}(t)\; dt, & t \in [0, T],\\
            X(0)=\xi, \label{evolution_equation1}
        \end{cases}
    \end{flalign}
    where \(A: \mathcal{D}(A) \to \mathcal{H}\) is the infinitesimal generator of a strongly continuous semigroup \(S(t)\), \(\xi\) is a \(\mathcal{F}_0\)-measurable \(\mathcal{H}\)-valued random variable and $\Phi: [0, T] \to \mathcal{H},\ \Phi \in L^2\left([0, T]; \mathcal{H}\right)$.
    \begin{definition}\label{def:mildsol}
        An $\mathcal{H}$-valued predictable stochastic process \(X(t)\), \(t \in [0, T]\) is said to be a mild solution to \eqref{evolution_equation1}, if for arbitrary \( t \in [0, T]\),
        \begin{flalign}
            X(t)=S(t)\xi+\int_{0}^tS(t-s)\Phi(s)\ dB^H(s)+\int_{0}^tS(t-s)\tilde{\Phi}(s)\; ds \quad \text{a.s.}\label{mild_sol}
        \end{flalign}
        where $S(t)$ is the strongly continuous semigroup generated by $A$. 
    \end{definition}
    \begin{definition}
        If $H=\frac{1}{2}$, a \(\mathcal{D}(A)\)-valued predictable stochastic process \(X(t)\), \(t \in [0, T]\) is said to be a strong solution to \eqref{evolution_equation1}, if for arbitrary \( t \in [0, T]\),
        \begin{flalign}
            X(t)=\xi+\int_{0}^t AX(s)ds+ \int_0^t \Phi(s)\ dW(s)+\int_0^t \tilde{\Phi}(s)\;ds \quad \text{a.s.} \label{strong_sol}
        \end{flalign}
    \end{definition}

    For simplicity, let $-\Delta$ denote the Laplace operator on $D$ with homogeneous Dirichlet boundary conditions. It is well-known that the operator $-\Delta$ has eigenvalues $\{\lambda_k\}^\infty_{k=1}$ and eigenfunctions $\{\varphi_k\}^\infty_{k=1}$, where the eigenvalues $\{\lambda_k\}^\infty_{k=1}$ satisfy $0<\lambda_1\le \lambda_2\le\cdots$ with $\lambda_k \to \infty$ as $k \to \infty$, and the eigenfunctions $\{\varphi_k\}^\infty_{k=1}$ form a basis in $L^2(D)$.
    
    Equip the space $\dot{H}^{s}(D):=\mathcal{D}\left((-\Delta)^{s/2}\right)$ with the norm
    \begin{equation}
        \|u\|_{\dot{H}^s(D)}=\left(\sum_{k=1}^\infty \lambda_k^s\langle u(x), \varphi_k(x) \rangle^2\right)^{1/2}.
    \end{equation}
    Clearly, $\|u\|_{\dot{H}^s(D)}$ and $ \|u\|_{H^s(D)}$ are equivalent norms on $\mathcal{D}\left((-\Delta)^{s/2}\right)$ when $s=1, 2$\cite[Lemma 3.1]{pub.1039910508}.

\section{The direct problem}
    \subsection{Well-posedness of stochastic MGT equation}
    Let $X(t) = [u(\cdot,t), u_t(\cdot,t), u_{tt}(\cdot,t)]^T$ and $\mathcal{H}=\dot{H}^2(D) \times \dot{H}^1(D)\times L^2(D)$, then equation \eqref{MGT} turns to
        \begin{equation}
            \begin{cases}
                dX(t) =
                AX(t)dt +\Phi(t) dB^H(t)+\tilde{\Phi}(t)dt,\\
                X(0)=0. \label{evolution_equation}
            \end{cases}
        \end{equation}
        where
        \begin{flalign*}
            A = \begin{bmatrix} 0 & I & 0 \\ 0 & 0 & I \\ c^2 \Delta & b \Delta & -\alpha \end{bmatrix}, \quad \Phi(t)=\begin{bmatrix}
            0\\ 0\\ f_1(\cdot)g_1(t) 
        \end{bmatrix}, \quad \text{and } \tilde{\Phi}(t)=\begin{bmatrix}
            0\\ 0\\ f_2(\cdot)g_2(t)
        \end{bmatrix}.
        \end{flalign*} The domain of the operator $A$ is defined as $\mathcal{D}(A) = \mathcal{D}(-\Delta) \times \mathcal{D}(-\Delta) \times \mathcal{D}(-\Delta)^{\frac{1}{2}}$. We can prove that $A$ generates a strongly continuous semigroup $e^{At}$ on $\mathcal{H}$, thus $A$ is a closed operator on $\mathcal{H}$\cite[Theorem 3.1]{pazy2012semigroups}, with an exponential decay rate $\omega$ when $\gamma >0$\cite[Theorem 2.2, 2.3]{articleExponentialDecay}, and
        \begin{equation}\label{eq:exponential decay}
             \|e^{At}\|_{\mathcal{L}(\mathcal{H})} \leq \mu e^{-\omega t}, \quad t \geq 0,
        \end{equation}
        where $\mu \geq 1$ and $\omega > 0$\cite[Theorem 3.1]{https://doi.org/10.1002/mma.1576}.

    To estimate the second moment of the stochastic convolution for the case $H > \frac{1}{2}$, we introduce the following inequality:
    \begin{lemma}\cite[Lemma 2.1]{fractionalBrownian}\label{le_inequality}
        If $p > 1/H$, then for any $\varphi \in L^p(0, T; \mathbb{R})$, the following inequality is satisfied:
	    \begin{equation}
            \int_0^T \int_0^T \varphi(u) \varphi(v) \phi(u-v) \, du dv \le C_T \|\varphi\|_{L^p(0, T; \mathbb{R})}^2
	    \end{equation}
	    for some constant $C_T > 0$ that only depends on $T$, where $\phi(u) = H(2H-1)|u|^{2H-2}$.
    \end{lemma}

We now turn to the proof of Theorem \ref{thm_existence}.
\begin{proof}
    Define the stochastic convolution $W_A(t)$ as follows:
    \begin{equation}\label{eq:W_A}
        W_A(t) := \int_0^t e^{A(t-s)} \Phi(s) \, dB^H(s).
    \end{equation}
    To establish the well-posedness of the system, it is crucial to show that the second moment of $W_A(t)$ is finite. Below, we discuss the cases $H=\frac{1}{2},H >\frac{1}{2},H < \frac{1}{2}$ separately, since the covariance operator of $B^H$ takes different forms in these three regimes.
    In the following proof, $C$ denotes a general constant whose value will change from line to line.
    
    \noindent \textbf{Case 1: $H = \frac{1}{2}$.} By \eqref{eq:exponential decay}, we have
    \begin{equation*}
        \norm{e^{A(t-s)}\Phi(s)}_{\cH}^2 \leq \mu^2 e^{-2\omega(t-s)}\norm{\Phi(s)}_{\cH}^2 \leq C\abs{g_1(s)}^2 \norm{f_1}_{L^2(D)}^2 .
    \end{equation*}
    Then it follows from It{\^o}'s isometry \eqref{eq:H=1/2} that 
        \begin{flalign*}
        \E\big[\abs{W_A(t)}^2\big] &= \int_0^t \|e^{A(t-s)} \Phi(s)\|_\mathcal{H}^2 ds  \\
        &\leq  \int_0^t \mu^2 e^{-2\omega(t-s)} \|\Phi(s)\|_\mathcal{H}^2 ds  \leq C\|f_1\|_{L^2(D)}^2 \|g_1\|_{L^2(0,T)}^2<\infty.
        \end{flalign*}
        
    \noindent \textbf{Case 2: $H > \frac{1}{2}$.} By applying \eqref{eq:H>1/2}, we obtain
        \begin{flalign*}
        \E\big[\abs{W_A(t)}^2\big] &= \int_0^t \int_0^t \inner{e^{A(t-s)}\Phi(s)}{e^{A(t-r)}\Phi(r)}_{\cH} \phi(r-s) \, ds \, dr \\
        &\leq C \int_0^t \int_0^t \norm{e^{A(t-s)}\Phi(s)}_{\cH} \norm{e^{A(t-r)}\Phi(r)}_{\cH} \phi(r-s) \, ds \, dr \\
        &\leq C \norm{f_1}^2_{L^2(D)}\int_0^t \int_0^t \abs{g_1(s)}  \abs{g_1(r)} \phi(r-s) \, ds \, dr.
        \end{flalign*}
    Since $g_1 \in L^2(0,T)$ for $H > \frac{1}{2}$, by Lemma \ref{le_inequality} we readily obtain that
    \begin{equation}
        \E\big[\abs{W_A(t)}^2\big] \leq C \norm{f_1}_{L^2(D)}^2 \norm{g_1}_{L^2(0,T)}^2 < \infty.
    \end{equation}

    \noindent \textbf{Case 3: $H < \frac{1}{2}$.} By \eqref{eq:H<1/2} and the expression of $K_H^*$ given in \eqref{eq:K_H^*}, we have the following estimate
    \begin{align}\label{eq:es-I1+I2}
        \E\big[\abs{W_A(t)}^2\big] &= \int_0^t \norm{K_H^*(t) \left(e^{A(t-\cdot)}\Phi(\cdot)\right)(s)}_{\cH}^2 \, ds \\
        &\leq C  \int_0^t \norm{K_H(t,s) e^{A(t-s)}\Phi(s)}_{\cH}^2 \, ds \\
        &\quad + C\int_0^t \left( \int_s^t \norm{e^{A(t-\tau)}\Phi(\tau) - e^{A(t-s)}\Phi(s)}_{\cH} \abs{\frac{\partial}{\partial \tau} K_H(\tau,s)} \, d\tau \right)^2 \, ds \\
        &=: I_1 + I_2.
    \end{align}
    Using \eqref{eq_kernel_1}, the first term $I_1$ can be estimated as
    \begin{align}\label{eq:es-I1}
        I_1 \leq C \norm{g_1}_{L^\infty(0,T)}^2 \norm{f_1}_{L^2(D)}^2 \int_0^t (K_H(t,s))^2 \, ds=C t^{2H} \norm{g_1}_{L^\infty(0,T)}^2 \norm{f_1}_{L^2(D)}^2.
    \end{align}
    To estimate $I_2$, we further decompose it into two parts:
    \begin{align*}
        I_2 &\leq C \int_0^t \left( \int_s^t \norm{\left(e^{A(t-\tau)} - e^{A(t-s)}\right)\Phi(\tau)}_{\cH} \abs{\frac{\partial}{\partial \tau} K_H(\tau,s)} \, d\tau \right)^2 ds \\
        &\quad + C \int_0^t \left( \int_s^t \norm{e^{A(t-s)}\big(\Phi(\tau) - \Phi(s)\big)}_{\cH} \abs{\frac{\partial}{\partial \tau} K_H(\tau,s)} \, d\tau \right)^2 ds \\
        &=: I_{21} + I_{22}.
    \end{align*}
    For $I_{21}$, since $\Phi(\tau) \in \mathcal{D}(A)$, it follows from standard semigroup theory (see \cite[Theorem 2.4(d)]{pazy2012semigroups}) that
    \begin{equation*}
        \left(e^{A(t-\tau)} - e^{A(t-s)}\right)\Phi(\tau) = g_1(\tau) \int_s^\tau -e^{A(t-r)} A \begin{bmatrix} 0 \\ 0 \\ f_1 \end{bmatrix} dr.
    \end{equation*}
    Using this relation and the fractional kernel derivative bound \eqref{eq_kernel_4}, $I_{21}$ can be bounded by
    \begin{align}\label{eq:es-I21}
        I_{21} &\leq C \int_0^t \left( \int_s^t \abs{g_1(\tau)} \int_s^\tau \norm{e^{A(t-r)} A \begin{bmatrix} 0 \\ 0 \\ f_1 \end{bmatrix}}_{\cH} dr \, \abs{\frac{\partial}{\partial \tau} K_H(\tau,s)} d\tau \right)^2 ds \nonumber\\
        &\leq C \int_0^t \left( \int_s^t \norm{f_1}_{H_0^1(D)} \abs{g_1(\tau)} \abs{\frac{\partial}{\partial \tau} K_H(\tau,s)} (\tau-s) \, d\tau \right)^2 ds \nonumber\\
        &\leq C \norm{f_1}_{H_0^1(D)}^2 \norm{g_1}_{L^\infty(0,T)}^2 \int_0^t \left( \int_s^t (\tau-s)^{H-\frac{1}{2}} \, d\tau \right)^2 ds \nonumber\\
        &\leq C t^{2H+2} \norm{f_1}_{H_0^1(D)}^2 \norm{g_1}^2_{L^\infty(0,T)} .
    \end{align}
    For $I_{22}$, since $g_1 \in C^{0,\rho}([0,T])$ with $\rho > \frac{1}{2} - H$, we obtain
    \begin{equation*}
        \norm{\Phi(\tau) - \Phi(s)}_{\cH} \leq \abs{\tau-s}^\rho \norm{f_1}_{L^2(D)} \norm{g_1}_{C^{0,\rho}([0,T])}
    \end{equation*}
    Substituting this into $I_{22}$ yields:
    \begin{align}\label{eq:es-I22}
        I_{22} &\leq C \int_0^t \left( \int_s^t \abs{\tau-s}^\rho \norm{f_1}_{L^2(D)} \norm{g_1}_{C^{0,\rho}([0,T])} \abs{\frac{\partial}{\partial \tau} K_H(\tau,s)} d\tau \right)^2 ds \nonumber\\
        &\leq C \norm{f_1}_{L^2(D)}^2 \norm{g_1}_{C^{0,\rho}([0,T])}^2 \int_0^t \left( \int_s^t (\tau-s)^{\rho+H-\frac{3}{2}} \, d\tau \right)^2 ds \nonumber\\
        &\leq C t^{2H+2\rho} \norm{f_1}_{L^2(D)}^2 \norm{g_1}_{C^{0,\rho}([0,T])}^2.
    \end{align}
    Combining \eqref{eq:es-I1},\eqref{eq:es-I21} and \eqref{eq:es-I22}, we obtain for $H<\frac{1}{2}$ that $\E\big[\abs{W_A(t)}^2\big] < \infty$ for all $t \in [0,T]$.
    
    Therefore, the mean-square integrability is guaranteed, which implies that the equation \eqref{evolution_equation} admits a mild solution. As for the uniqueness, it follows directly from the explicit expression \eqref{eq:mild-solution}, which completes the proof.
    \end{proof}
    \begin{remark}
    When $H = \frac{1}{2}$, the fractional Brownian motion $B^H$ reduces to a standard Brownian motion. In this case, if $f_i \in H^2(D) \cap H_0^1(D) $ for $i=1, 2$, we can prove the unique mild solution $X(t)$ is also a strong solution of the equation \eqref{evolution_equation}. Furthermore, we have the estimate
    \begin{flalign}
        \mathbb{E} [\|X(s)\|^2_{\mathcal{H}}] \le C\left(\|f_1\|^2_{H^2(D)}\|g_1\|^2_{L^2(0, T)}+\|f_2\|^2_{L^2(D)}\|g_2\|^2_{L^2(0, T)}\right). \label{energy_estimate}
    \end{flalign}
    \end{remark}
    
    \begin{theorem}\label{tmh_mean_square_conti}
        Suppose that the assumptions of Theorem \ref{thm_existence} hold. The unique mild solution $X(t)$ of the stochastic MGT equation \eqref{evolution_equation} given by \eqref{eq:mild-solution} is mean-square continuous.
    \end{theorem}
    \begin{proof}
         By standard $C_0$-semigroup theory and the integrability of $\tilde{\Phi}$, the deterministic integral component of the mild solution \eqref{eq:mild-solution} is mean-square continuous. Thus, it remains only to prove the mean-square continuity of the stochastic convolution \eqref{eq:W_A}.
        
        For $h>0$ sufficiently small such that $t+h \le T$, we have
        \begin{flalign*}
            &\mathbb{E}\left[\left\| W_A(t+h) - W_A(t) \right\|_{\mathcal{H}}^2\right]\\
    \leq &C \mathbb{E}\left[ \left\| \int_0^t \left( e^{A(t+h-s)}-e^{A(t-s)}\right) \Phi(s)\,\mathrm{d}B^H(s) \right\|_{\mathcal H}^2 \right]+ C\mathbb{E} \left[ \left\| \int_t^{t+h} e^{A(t+h-s)}\Phi(s)\,\mathrm{d}B^H(s) \right\|_{\mathcal H}^2 \right] \\
    =: &C (K_1+K_2).
    \end{flalign*}
    
    For the term $K_2$, regardless of $H \in (0,1)$, we can define $\tilde{B}^H(s) = B^H(t+s) - B^H(t)$. By the stationarity of the increments of fractional Brownian motion, $\tilde{B}^H$ is a fractional Brownian motion with the same Hurst parameter. Changing the integration variable \(s\mapsto t+s\) yields
    \begin{flalign*}
        K_2 = \mathbb{E}\left[\left\| \int_0^h e^{A(h-s)}\Phi(t+s)\,\mathrm{d}\widetilde{B}^H(s)\right\|_{\mathcal H}^2\right]=\int_0^h \left\|K^*_H(h)\left(e^{A(h-\cdot)}\Phi(t+\cdot)\right)(s)\right\|^2_{\mathcal{H}} \;ds.
    \end{flalign*}
    Applying the finite second moment bounds \eqref{eq:es-I1+I2} established in \autoref{thm_existence} on the interval $[0, h]$, the absolute continuity of the Lebesgue integral ensures $\lim_{h\to 0} K_2 = 0$.

    Next, we estimate $K_1$ depending on the Hurst parameter $H$.\\
    \noindent\textbf{Case 1: \(H \ge \frac12\).} By the Itô isometry \eqref{eq:H=1/2} for $H=1/2$ and the covariance isometry \eqref{eq:H>1/2} for $H>1/2$, the term $K_1$ is bounded by integrals involving $\|(e^{A(t+h-s)} - e^{A(t-s)})\Phi(s)\|_{\mathcal{H}}$. For every fixed $s \in [0,t]$, the strong continuity of the semigroup yields
    \begin{flalign*}
        \lim_{h\to 0}\|(e^{A(t+h-s)} - e^{A(t-s)})\Phi(s)\|_{\mathcal{H}} = 0
    \end{flalign*}
    Since the semigroup is uniformly bounded, the integrand norm satisfies:
    \begin{flalign}\label{eq:es-conbound}
        \left\|\left(e^{A(t+h-s)}-e^{A(t-s)}\right)\Phi(s)\right\|_{\mathcal H}^2\leq C|g_1(s)|^2\|f_1\|_{H^1(D)}^2.
    \end{flalign}
    Therefore, the Dominated Convergence Theorem implies $\lim_{h\to 0} K_1 = 0$.

    \noindent\textbf{Case 2: \(H < \frac12\).} Similar to the estimate \eqref{eq:es-I1+I2}, we can decompose $K_1$ as
    \begin{align*}
        K_1 &\le C \int_0^t K_H(t,s)^2 \left\| \left(e^{A(t+h-s)} - e^{A(t-s)}\right) \Phi(s) \right\|_{\mathcal{H}}^2 ds \\
         &\quad + C \int_0^t \left( \int_s^t \left\| \left(e^{A(t+h-s)} - e^{A(t-s)}\right)\Phi(s) - \left(e^{A(t+h-\tau)} - e^{A(t-\tau)}\right)\Phi(\tau) \right\|_{\mathcal{H}} \left| \frac{\partial K_H}{\partial \tau}(\tau,s) \right| d\tau \right)^2 ds \\
        &=: K_{11} + K_{12}.
    \end{align*} 
    For $K_{11}$, by \eqref{eq:es-conbound}, the integrand is also dominated by an integrable function:
    \begin{flalign*}
        K_H(t,s)^2 \left\| \left(e^{A(t+h-s)} - e^{A(t-s)}\right) \Phi(s) \right\|_{\mathcal{H}}^2 \leq  C K_H(t,s)^2 \|g_1\|_{L^\infty(0,T)}^2 \|f_1\|_{H^1(D)}^2,
    \end{flalign*}
    since $\int_0^t K_H(t,s)^2 ds = C t^{2H} < \infty$. Then the Dominated Convergence Theorem similarly yields $\lim_{h\to 0} K_{11} = 0$. \\
    For $K_{12}$, we decompose the operator differences to obtain
    \begin{flalign*}
        K_{12}\leq &C\int_0^t \left( \int_s^t \left\| \left( e^{A(t+h-s)}-e^{A(t-s)} \right) \bigl(\Phi(s)-\Phi(\tau)\bigr) \right\|_{\mathcal H} \left| \frac{\partial K_H}{\partial\tau}(\tau,s) \right| \,\mathrm{d}\tau \right)^2 \,\mathrm{d}s\\
        &+ C\int_0^t \left( \int_s^t \left\| \left(e^{Ah}-I \right) \left( e^{A(t-s)}-e^{A(t-\tau)} \right)\Phi(\tau) \right\|_{\mathcal H} \left| \frac{\partial K_H}{\partial\tau}(\tau,s) \right| \,\mathrm{d}\tau \right)^2 \,\mathrm{d}s.
    \end{flalign*}
    For the first term, since $g_1 \in C^{0,\rho}([0,T])$, we obtain
    \begin{align*}
        \left\| \left(e^{A(t+h-s)} - e^{A(t-s)}\right)\left(\Phi(s) - \Phi(\tau)\right) \right\|_{\mathcal{H}} \le C h \, |s-\tau|^{\rho} \, \|f_1\|_{H^1(D)}.
    \end{align*}
    For the second term, utilizing the standard bounds of the semigroup $e^{At}$, the term is bounded by both $C|g_1(\tau)|\,|\tau-s|\,\|f_1\|_{H^1(D)}$ and $C h\,|g_1(\tau)|\,\|f_1\|_{H^1(D)}$.
    Choosing $\theta$ such that $\frac{1}{2} - H < \theta < 1$ and using the interpolation $\min\{h, |\tau-s|\} \le h^{1-\theta} |\tau-s|^\theta$, we deduce
    \begin{align*}
        \left\| \left(e^{Ah} - I\right)\left(e^{A(t-s)} - e^{A(t-\tau)}\right)\Phi(\tau) \right\|_{\mathcal{H}} \le C h^{1-\theta} |\tau-s|^\theta |g_1(\tau)| \, \|f_1\|_{H^1(D)}.
    \end{align*}
    Then using the estimate \eqref{eq_kernel_4}, we arrive at
    \begin{align*}
        K_{12} &\le C h^2 \left\| f_1 \right\|_{H^1(D)}^2 \left\| g_1 \right\|_{C^{0,\rho}([0,T])}^2 \int_0^t \left( \int_s^t |\tau - s|^{\rho
        + H - \frac{3}{2}} \,d\tau \right)^2 ds \\
        &\quad + C h^{2-2\theta} \left\| f_1 \right\|_{H^1(D)}^2 \left\| g_1 \right\|_{L^\infty(0,T)}^2 \int_0^t \left( \int_s^t |\tau - s|^{\theta + H - \frac{3}{2}} \,d\tau \right)^2 ds \\
        &\le C h^2 \left\| f_1 \right\|_{H^1(D)}^2 \left\| g_1 \right\|_{C^{0,\rho}([0,T])}^2 t^{2\rho + 2H} + C h^{2-2\theta} \left\| f_1 \right\|_{H^1(D)}^2 \left\| g_1 \right\|_{L^\infty(0,T)}^2 t^{2\theta + 2H}.
    \end{align*}
    Both terms strictly vanish as $h \to 0$, hence $\lim_{h\to 0} K_{12} = 0$.

    Consequently, $\lim_{h\to 0} \mathbb{E}[\| W_A(t+h) - W_A(t) \|_{\mathcal{H}}^2] = 0$, proving the mean-square continuity. 
    \end{proof}

Since the mild solution $X(t)$ is mean-square continuous and adapted to the filtration generated by the fractional Brownian motion, it admits a predictable modification. This rigorously satisfies the predictability requirement for the mild solution stated in Definition \ref{def:mildsol}. 

Now we give a useful representation of the mild solution. 
By the property of $C_0$-semigroup, $e^{A(t-s)}\Phi(s, x)$ is also a solution of a deterministic MGT equation.
\begin{theorem}\label{thmexplicit_form__mildsolution}
    Suppose that the assumptions of Theorem \ref{thm_existence} hold. For $H \in (0,1)$ and $T>0$, the unique mild solution $u(t)$ of the equation \eqref{MGT} can be written as 
\begin{align*}
	   u(x,t) = \int_0^t v_1(x,t-s) g_1(s) dB^H(s)+\int_0^t v_2(x,t-s)g_2(s) ds,
    \end{align*}
    where for $i=1,2$, $v_i(x, t)$ is the solution to 
    \begin{equation}
        \begin{cases}
		      \partial^3_{t} v_i + \alpha \partial^2_t v_i - c^2 \Delta v_i - b \Delta \partial_t v_i = 0 & D \times (0,T), \\
		      v_i = 0 & \partial D \times (0,T), \\
		      v_i(x,0) = 0, \quad \partial_t v_i(x,0) = 0, \quad \partial_t^2 v_i(x,0) = f_i(x) & D.
        \end{cases} \label{MGT_2}
    \end{equation}
    \end{theorem}
\subsection{Hidden regularity}
\noindent
We can observe that the inverse problem considered in this paper requires the normal derivative of the solution to be defined on the boundary. Since stochastic wave equation satisfies certain extra regularity\cite[Proposition 3.1]{L2011GlobalUF}, it's natural to expect the Moore-Gibson-Thompson equation also satisfies an analogous result. 

\begin{lemma} \label{lemma_exchange}
    Let $E_1$ and $E_2$ be separable Hilbert spaces. Let $\mathcal{A}: E_1 \to E_2$ be a bounded linear operator. If $\Psi$ is a deterministic function such that $\Psi \in \mathcal{H}_0([0,T]; E_1)$, then $\int_0^T \Psi(s) dB^H(s) \in E_1$ $\mathbb{P}$-a.s., and
    \begin{equation}\label{eq:change}
        \mathcal{A} \int_0^T \Psi(s) d B^H(s) = \int_0^T \mathcal{A} \Psi(s) d B^H(s) \quad \mathbb{P}\text{-a.s.}
    \end{equation}
\end{lemma}
\begin{proof}
    We first prove the result for $E_1$-valued step functions. Suppose $\Psi(s)$ has the form
    \begin{equation*}
        \Psi(s) = \sum_{j=0}^{m-1} a_j \mathbf{1}_{(t_j, t_{j+1}]}(s), \quad a_j \in  E_1.
    \end{equation*}
    By the definition \eqref{eq:defWiener}, it trivially holds that
    \begin{align*}
    \mathcal{A} \int_0^T \Psi(s) d B^H(s) = \mathcal{A} \sum_{j=0}^{m-1} a_j (B^H(t_{j+1}) - B^H(t_j)) = \int_0^T \mathcal{A} \Psi(s) d B^H(s).
    \end{align*}
    
    For a general deterministic function $\Psi \in \mathcal{H}_0([0,T]; E_1)$, by the density of step functions, there exists a sequence of $E_1$-valued step functions $\{\Psi_n\}$ such that 
    \begin{equation*}
        \Psi_n \to \Psi \quad \text{in } \mathcal{H}_0([0,T]; E_1).
    \end{equation*}
    As $\mathcal{A}: E_1 \to E_2$ is a bounded linear operator, it induces a continuous mapping from $\mathcal{H}_0([0,T]; E_1)$ to $\mathcal{H}_0([0,T]; E_2)$, which directly implies $\mathcal{A}\Psi_n \to \mathcal{A}\Psi$ in $\mathcal{H}_0([0,T]; E_2)$. By the isometry property of the fractional Wiener integral, we have
    \begin{align*}
        \int_0^T \Psi_n(s) dB^H(s) \to \int_0^T \Psi(s) dB^H(s) \quad \text{in } L^2(\Omega; E_1), \\
        \int_0^T \mathcal{A}\Psi_n(s) dB^H(s) \to \int_0^T \mathcal{A}\Psi(s) dB^H(s) \quad \text{in } L^2(\Omega; E_2).
    \end{align*}
    By the continuity of $\mathcal{A}$ and the uniqueness of limits in $L^2(\Omega; E_2)$, applying $\mathcal{A}$ to the first limit immediately yields
    \begin{align*}
        \mathcal{A}\int_0^T \Psi(s) dB^H(s) = \int_0^T \mathcal{A}\Psi(s) dB^H(s) \quad \text{in } L^2(\Omega; E_2).
    \end{align*}
    Equality in the $L^2(\Omega;E_2)$ implies $\mathbb{P}$-a.s. equality. This completes the proof. 
\end{proof}
The next proposition follows from the special case $m=0, 1$ of \cite[Theorem 3.1]{Fu_2024}.
\begin{proposition}\label{thm_hiddenregularity}
Assume $\partial D$ is $C^\infty$ and $v$ is the unique solution of 
\begin{equation*}
	\begin{cases}
		\partial^3_{t} v + \alpha \partial^2_t v - c^2 \Delta v - b \Delta \partial_t v = 0 & D \times (0,T), \\
		v = 0 & \partial D \times (0,T), \\
		v(x,0) = 0, \quad \partial_t v(x,0) = 0, \quad \partial_t^2 v(x,0) = f(x) & D.
	\end{cases}
\end{equation*}
If $f \in H^2(D) \cap H_0^1(D)$, then $v \in C^2\left([0, T]; H^2(D) \cap H_0^1(D)\right)$, $\partial_{\nu} v \in H^{3}\left(0, T; L^2(\partial D)\right)$ and 
\begin{equation}\label{es-m2}
    \|\partial_{\nu} v\|_{H^{3}\left(0, T; L^2(\partial D)\right)} \le C\|f\|_{H^2(D)}.
\end{equation}
If $f \in H^1_0(D)$, then $v \in C^1\left([0, T]; H^2(D) \cap H_0^1(D)\right)$, $\partial_{\nu} v \in H^{2}\left(0, T; L^2(\partial D)\right)$ and 
\begin{equation}\label{es-m1}
    \|\partial_{\nu} v\|_{H^{2}\left(0, T; L^2(\partial D)\right)} \le C\|f\|_{H^1(D)}.
\end{equation}
If $f \in L^2(D)$, then $v \in C\left([0, T]; H^2(D) \cap H_0^1(D)\right)$, $\partial_{\nu} v \in H^{1}\left(0, T; L^2(\partial D)\right)$ and 
\begin{equation}\label{es-m0}
    \|\partial_{\nu} v\|_{H^{1}\left(0, T; L^2(\partial D)\right)} \le C\|f\|_{L^2(D)}.
\end{equation}
\end{proposition}

\begin{theorem}\label{thm_hidden_regularity}
    Suppose that the assumptions of Theorem \ref{thm_existence} hold. Additionally, assume the boundary $\partial D$ is $C^\infty$, then the unique mild solution $(u, u_t, u_{tt})$ of equation \eqref{MGT} satisfies
    \begin{equation}\label{es-H>1/2}
        \text{sup}_{0 \le t \le T}\|\partial_{\nu} u(t)\|_{L^2\left( \Omega; L^2(\partial D)\right)} \le C\left(\|f_1\|_{H^1(D)}\|g_1\|_{L^2([0, T])}+\|f_2\|_{L^2(D)}\|g_2\|_{L^2([0, T])}\right),  
    \end{equation} 
    when $H \ge \frac{1}{2}$ and 
    \begin{equation}\label{es-H<1/2}
    \text{sup}_{0 \le t \le T}\|\partial_{\nu} u(t)\|_{L^2\left( \Omega; L^2(\partial D)\right)}  \le C \left( \|f_1\|_{H^1(D)}\|g_1\|_{C^{0, \rho}([0, T])} + \|f_2\|_{L^2(D)}\|g_2\|_{L^2(0,T)} \right),
    \end{equation}
    when $H<\frac{1}{2}$.
\end{theorem}
\begin{proof}
    In the following proof, $C$ denotes a general constant that may change value from line to line. 
    From Theorem \ref{thmexplicit_form__mildsolution}, the mild solution can be written as
    \begin{align*}
	   u(x,t) = \int_0^t v_1(x,t-s) g_1(s) dB^H(s)+\int_0^t v_2(x,t-s)g_2(s) ds,
    \end{align*}
    where $v_i$ is the solution of the equation \eqref{MGT_2}. \\
    By Proposition \ref{thm_hiddenregularity} and \cite[Chapter 5.9.2, Theorem 2]{evans2022partial}, we have $\partial_\nu v_1 \in H^2(0, T; L^2(\partial D)) \hookrightarrow C^{0,1}([0, T]; L^2(\partial D))$ and $\partial_\nu v_2 \in H^1(0, T; L^2(\partial D)) \hookrightarrow L^\infty(0, T; L^2(\partial D))$. Thus, both the map $s \mapsto v_1(\cdot, t - s)g_1(s)$ and its image $s \mapsto \partial_\nu v_2(\cdot, t - s)g_2(s)$ are Bochner integrable on $[0, t]$ with respect to their corresponding spaces. \\
    Similarly, the analogous regularity of $v_1$ and temporal smoothness of $g_1$ ensure the integrand $s \mapsto v_1(\cdot, t-s)g_1(s)$ belongs to the fractional Wiener space $\mathcal{H}_0([0,T]; H^2(D)\cap H_0^1(D))$. Since $\partial_{\nu}: H^2(D) \cap H_0^1(D)  \to H^{\frac{1}{2}}(\partial D)\hookrightarrow L^2(\partial D)$ is a bounded linear operator, by applying \cite[Theorem 3.7.12]{hille1996functional} and Lemma \ref{lemma_exchange}, we obtain
    \begin{equation*}
	   \partial_{\nu} u(x,t) = \int_0^t \partial_{\nu} v_1(x, t-s) g_1(s)\; dB^H(s)+\int_0^t \partial_{\nu} v_2(x, t-s) g_2(s)\;ds.
    \end{equation*}
    We proceed to estimate the two terms separately.

    For the deterministic term, applying H\"older's inequality yields:
    \begin{flalign}\label{es-deter}
        \left\| \int_0^t\partial_{\nu}v_2(t-s)g_2(s)\; ds \right\|_{L^2(\partial D)} 
        &\le C\int_0^t \left\|\partial_{\nu}v_2(t-s)g_2(s)\right\|_{L^2(\partial D)}\; ds\nonumber\\
        &\le C\left(\int_0^t\left\|\partial_{\nu} v_2(t-s)\right\|^2_{L^2(\partial D)}\; ds \int_0^t|g_2(s)|^2\; ds\right)^{\frac{1}{2}}\nonumber\\
        &\le C\|\partial_{\nu}v_2\|_{L^2(0,T;L^2(\partial D))} \|g_2\|_{L^2(0,T)}.
    \end{flalign}
    
    For the stochastic term, we divide the proof into three cases based on the value of the Hurst parameter H. \\
    \noindent \textbf{Case 1: $H = \frac{1}{2}$.} By Itô's isometry \eqref{eq:H=1/2},
    \begin{flalign}\label{es-case1}
        \mathbb E \left\| \int_0^t \partial_{\nu}v_1(t-s)g_1(s)\,dB^H(s) \right\|_{L^2(\partial D)}^2 &= \int_0^t \|\partial_{\nu}v_1(t-s)\|_{L^2(\partial D)}^2 |g_1(s)|^2\,ds\nonumber\\
        &\le C\|\partial_{\nu}v_1\|^2_{L^{\infty}(0,T;L^2(\partial D))} \|g_1\|^2_{L^2(0,T)}.
    \end{flalign}
    \noindent \textbf{Case 2: $H > \frac{1}{2}$.} By applying \eqref{eq:H>1/2} and Cauchy-Schwarz inequality, we obtain
    \begin{flalign}\label{es-case2}
        &\mathbb E \left\| \int_0^t \partial_{\nu}v_1(t-s)g_1(s)\,dB^H(s) \right\|_{L^2(\partial D)}^2 \nonumber\\
        &\le H(2H-1) \int_0^t\int_0^t  \|\partial_{\nu}v_1(t-s)\|_{L^2(\partial D)} \|\partial_{\nu}v_1(t-r)\|_{L^2(\partial D)} |g_1(s)||g_1(r)||s-r|^{2H-2}\,dsdr\nonumber\\
        &\le CH(2H-1)\|\partial_{\nu}v_1\|^2_{L^{\infty}\left(0, T;L^2(\partial D)\right)}\int_0^t\int_0^t  |g_1(s)||g_1(r)||s-r|^{2H-2}\,dsdr
        \nonumber\\
        &\le C_{H,T}\|\partial_{\nu}v_1\|^2_{L^{\infty}\left(0, T;L^2(\partial D)\right)}\|g_1\|^2_{L^2(0, T)},
    \end{flalign}
where the last step follows from Lemma \ref{le_inequality} for $g_1 \in L^2(0,T)$. \\
\noindent \textbf{Case 3: $H < \frac{1}{2}$.} By \eqref{eq:H<1/2} and the expression of $K_H^*$ given in \eqref{eq:K_H^*}, we have the following estimate
\begin{flalign*}
	&\mathbb{E}  \Big\| \int_0^t \partial_\nu v_1(t-s) g_1(s) d B^H(s) \Big\|_{L^2(\partial D)}^2 \Big] 
	= \int_0^t \big\| \left(K_H^*(t) \partial_\nu v_1(t-\cdot) g_1(\cdot)\right) (s) \big\|_{L^2(\partial D)}^2 ds \\
	&\le C \int_0^t \| K_H(t,s) \partial_\nu v_1(t-s) g_1(s) \|_{L^2(\partial D)}^2 ds \\
	&\quad + \int_0^t \left( \int_s^t \| \partial_\nu v_1(t-\tau)g_1(\tau) - \partial_\nu v_1(t-s) g_1(s) \|_{L^2(\partial D)} \Big| \frac{\partial}{\partial \tau} K_H(\tau, s) \Big| \;d\tau\right)^2 \;ds \\
	&=: J_1 + J_2.
\end{flalign*}
Similarly to \eqref{eq:es-I1}, the first term $J_1$ above can be estimated as follows:
\begin{align}\label{eq:es-J1}
J_1 &= \int_0^t \left\| K_H(t,s) \partial_\nu v_1(t-s) g_1(s) \right\|_{L^2(\partial D)}^2 ds \nonumber\\
&\le \left\| \partial_\nu v_1 \right\|_{L^\infty(0,T; L^2(\partial D))}^2 \left\| g_1 \right\|_{L^\infty(0,T)}^2 \int_0^t \big(K_H(t,s)\big)^2 ds \nonumber\\
&= C t^{2H} \left\| \partial_\nu v_1 \right\|_{L^\infty(0,T; L^2(\partial D))}^2 \left\| g_1 \right\|_{L^\infty(0,T)}^2.
\end{align}
Then using inequality \eqref{eq_kernel_4} for the partial derivative of the kernel, we split $J_2$ into two terms $J_2 \le C(J_{21} + J_{22})$, where
\begin{align*}
J_{21} &= \int_0^t \left( \int_s^t (\tau - s)^{H-\frac{3}{2}} \left\| \partial_\nu v_1(t-\tau) \right\|_{L^2(\partial D)} \left| g_1(\tau) - g_1(s) \right| \,d\tau \right)^2 ds, \\
J_{22} &= \int_0^t \left( \int_s^t (\tau - s)^{H-\frac{3}{2}} \left| g_1(s) \right| \left\| \partial_\nu v_1(t-\tau) - \partial_\nu v_1(t-s) \right\|_{L^2(\partial D)} \,d\tau \right)^2 ds.
\end{align*}
Since $g_1 \in C^{0,\rho}([0,T])$ with $\rho > \frac{1}{2} - H$, for $J_{21}$ we obtain:
\begin{align}\label{eq:es_J21}
J_{21} &\le C \left\| \partial_\nu v_1 \right\|_{L^\infty(0,T; L^2(\partial D))}^2 \left\| g_1 \right\|_{C^{0,\rho}([0,T])}^2 \int_0^t \left( \int_s^t (\tau - s)^{H-\frac{3}{2}+\rho} \,d\tau \right)^2 ds \nonumber\\
&\le C \left\| \partial_\nu v_1 \right\|_{L^\infty(0,T; L^2(\partial D))}^2 \left\| g_1 \right\|_{C^{0,\rho}([0,T])}^2 t^{2H+2\rho}.
\end{align}
For $J_{22}$, since $\partial_\nu v_1$ is Lipschitz continuous in time, for any $0 \le s \le t \le T$, we have
\begin{equation*}
    \left\| \partial_\nu v_1(t) - \partial_\nu v_1(s) \right\|_{L^2(\partial D)} 
\le C(t-s) \left\| \partial_\nu v_1 \right\|_{H^2(0,T; L^2(\partial D))}.
\end{equation*}
It follows that
\begin{align}\label{eq:es-J22}
J_{22} &\le C \left\| g_1 \right\|_{L^\infty(0,T)}^2 \int_0^t \left( \int_s^t (\tau - s)^{H-\frac{3}{2}} (\tau - s) \left\| \partial_\nu v_1 \right\|_{H^2(0,T; L^2(\partial D))} \,d\tau \right)^2 ds \nonumber\\
&\le C \left\| g_1 \right\|_{L^\infty(0,T)}^2 \left\| \partial_\nu v_1 \right\|_{H^2(0,T; L^2(\partial D))}^2 \int_0^t \left( \int_s^t (\tau - s)^{H-\frac{1}{2}} \,d\tau \right)^2 ds \nonumber\\
&\le C \left\| g_1 \right\|_{L^\infty(0,T)}^2 \left\| \partial_\nu v_1 \right\|_{H^2(0,T; L^2(\partial D))}^2 t^{2H+2}.
\end{align}
Combining the estimates \eqref{eq:es-J1}, \eqref{eq:es_J21} and \eqref{eq:es-J22}, we obtain:
\begin{equation}\label{es-case3}
    \mathbb{E}  \Big\| \int_0^t \partial_\nu v_1(t-s) g_1(s) d B^H(s) \Big\|_{L^2(\partial D)}^2   \le C t^{2H} \left\| \partial_\nu v_1 \right\|_{H^2(0,T; L^2(\partial D))}^2 \left\| g_1 \right\|_{C^{0,\rho}([0,T])}^2.
\end{equation}

Taking the supremum over $t \in [0,T]$ for the estimates \eqref{es-deter}, \eqref{es-case1}, \eqref{es-case2} and \eqref{es-case3}, and applying the bounds \eqref{es-m0} and \eqref{es-m1}, the final results follow immediately. 
\end{proof}

\section{Uniqueness of inverse random source problem}
\subsection{Uniqueness of recovering \texorpdfstring{$f_i(x)$}{fi(x)}}
In this section, we prove the uniqueness of recovering $f_i(x)$ when $H \in (0, 1)$ assuming the time function $g_i(t)$ are already known, which is inspired by \cite{Lassas_2023}. The proof of uniquely recovering $f_i$ is mainly based on the unique continuation property of Moore-Gibson-Thompson equation when the initial condition $v(x, 0)=v_t(x, 0)=0$ is satisfied. Before proving the results, we introduce the unique continuation property of MGT equation using Carleman estimate.

\begin{lemma}[Carleman estimate for MGT equation]\cite[Theorem 2.6]{article_Arancibia_2022} \label{le_Carleman_for_wave}
    Assume the time condition $T > \sup_{x \in D} |x - x_0|$ and the geometric condition 
    \begin{equation}
        \exists x_0 \notin \overline{D}, \text{ such that } \Gamma \supset \{x \in \partial D, \ (x - x_0) \cdot \nu(x) \geq 0\}, \label{geometric}
    \end{equation}
    are satisfied. Let $\beta \in (0,1)$ be such that 
    \begin{equation}
	   \sup_{x \in D} |x_0 - x| < \beta T.
    \end{equation}
    Let $\varphi_{\lambda} = e^{\lambda \psi(x,t)}$, where
    \begin{equation}
	   \psi(x,t) = |x - x_0|^2 - \beta t^2 + C_0,
    \end{equation}
    for some constant $C_0$ such that $\psi \ge 1$.
    Then, there exists $s_0 > 0, \lambda > 0$ and a positive constant $C$ such that
    \begin{equation}
    \begin{aligned}
        &\sqrt{s} \int_{D} e^{2s\varphi_{\lambda}(\cdot, 0)} |y_{tt}(\cdot, 0)|^2 dx + s\lambda c^4 \int_0^T \int_{D} e^{2s\varphi_{\lambda}} \varphi_{\lambda}(|y_t|^2 + |\nabla y|^2)dxdt \\
        &+ s^3\lambda^3 c^4 \int_0^T \int_{D} e^{2s\varphi_{\lambda}} \varphi_{\lambda}^3|y|^2 dxdt + s\lambda \int_0^T \int_{D} e^{2s\varphi_{\lambda}} \varphi_{\lambda}(|y_{tt}|^2 + |\nabla y_t|^2)dxdt \\
        &+ s^3\lambda^3 \int_0^T \int_{D} e^{2s\varphi_{\lambda}} \varphi_{\lambda}^3|y_t|^2 dxdt\\
        &\le C \int_0^T \int_{D} e^{2s\varphi_{\lambda}}|Ly|^2 dxdt + Cs\lambda \int_0^T \int_{\Gamma} e^{2s\varphi_{\lambda}} (|\nabla y_t \cdot \nu|^2 + c^4|\nabla y \cdot \nu|^2) d\sigma dt,
    \end{aligned}
    \end{equation}
    for all $s \ge s_0$ and for all $y \in L^2(0, T; H_0^1(D))$ satisfying $Ly := y_{ttt} + \alpha y_{tt} - c^2\Delta y - \Delta y_t \in L^2(D \times (0, T))$, $y(\cdot, 0) = y_t(\cdot, 0) = 0$ in $D$, and $y_{tt}(\cdot, 0) \in L^2(D)$.
\end{lemma}
\begin{theorem}\label{thm_UCP}
    If $\Gamma$ satisfies the geometric condition \eqref{geometric}, $T > \frac{1}{\sqrt{b}}\sup_{x \in D} |x - x_0|$ and $v$ is a solution of 
    \begin{equation}
	\begin{cases}
		\partial^3_{t} v + \alpha \partial^2_t v - c^2 \Delta v - b \Delta \partial_t v = 0  & D \times (0,T), \\
		v = 0 \quad \partial_{\nu}v=0& \Gamma \times (0,T), \\
		v(x,0) = 0, \quad \partial_t v(x,0) = 0, \quad \partial_t^2 v(x,0) = \phi(x) & D. \label{eq_determinsitic_MGT_for_control}
	\end{cases} 
\end{equation}
where $\phi \in L^2(D)$, then $\phi(x)=0$ a.e. in $D$.
\end{theorem}
\begin{proof}
Since $\phi(x) \in L^2(D)$, the regularity of the
solution of the MGT equation tells that $v \in C^1\left([0, T]; H_0^1(D) \right)\cap C^2\left([0, T]; L^2(D)\right)$\cite[Theorem 1.1]{Bucci2020}. Let $\tau=\sqrt{b}t$, $\tilde{v}(x, \tau)=v\left(x, t\right)$. Let $T_1=\sqrt{b}T$, then $T_1>\sup_{x \in D} |x - x_0|$ and $\tilde{v}(x, \tau)$ satisfies
    \begin{equation}
	   \begin{cases}
	       	\partial^3_{\tau} \tilde{v} + \frac{\alpha}{\sqrt{b}} \partial^2_{\tau} \tilde{v} - \frac{c^2}{\sqrt{b^3}} \Delta \tilde{v} - \Delta \partial_{\tau} \tilde{v} = 0  & D \times (0,T_1), \\
		    \tilde{v} = 0 \quad \partial_{\nu} \tilde{v}=0& \Gamma \times (0,T_1), \\
		      \tilde{v}(x,0) = 0, \quad \tilde{v}_{\tau}(x,0) =0 \quad \tilde{v}_{\tau \tau}(x, 0)=\frac{1}{b}\phi(x)  & D.
	   \end{cases}
    \end{equation}
    Since $\partial_{\nu} \tilde{v}=0$ in $H^1\left((0, T_1);L^2(\Gamma)\right)$ and $\partial_{\tau}(\partial_{\nu} \tilde{v})=\partial_{\nu}\tilde{v}_{\tau}$, we obtain $\partial_{\nu} \tilde{v}_{\tau}=0$ in $L^2\left((0, T_1);L^2(\Gamma)\right)$. Applying Lemma \ref{le_Carleman_for_wave} for $\tilde{v}$, whose interior and boundary terms vanish, yields 
    \begin{flalign*}
        &\sqrt{s} \int_{D} e^{2s\varphi_{\lambda}(\cdot, 0)} |\tilde{v}_{\tau \tau}(\cdot, 0)|^2 dx=\frac{\sqrt{s}}{b^2} \int_{D} e^{2s\varphi_{\lambda}(\cdot, 0)} |\phi(\cdot)|^2 dx\\
        &\le C \int_0^{T_1} \int_{D} e^{2s\varphi_{\lambda}}|L\tilde{v}|^2 dxd\tau + Cs\lambda \int_0^{T_1} \int_{\Gamma} e^{2s\varphi_{\lambda}} (|\nabla \tilde{v}_{\tau} \cdot \nu|^2 + \frac{c^4}{b^3}|\nabla \tilde{v} \cdot \nu|^2) d\sigma d\tau=0
    \end{flalign*}
    Therefore, $\phi(x)=0$ a.e. in $D$.   
\end{proof}
We now turn to the proof of Theorem \ref{thm_unique-f}.
\begin{proof}
    By the proof of \autoref{thm_hidden_regularity}, for an arbitrary test function $\eta \in C^{\infty}_c(\Gamma)$, we have
    \begin{equation*}
	   \langle\partial_{\nu} u(t) ,\eta\rangle_{L^2(\Gamma)} = \int_0^t \langle\partial_{\nu} v_1(t-s),\eta\rangle_{L^2(\Gamma)} g_1(s)\; dB^H(s)+\int_0^t \langle\partial_{\nu} v_2(t-s),\eta\rangle_{L^2(\Gamma)} g_2(s)\;ds.
    \end{equation*}
    First taking the expectation yields
    \begin{equation} \label{expectation}
        \mathbb{E}\left[ \langle\partial_{\nu} u(t) ,\eta\rangle_{L^2(\Gamma)}\right]=\int_0^t \langle\partial_{\nu} v_2(t-s),\eta\rangle_{L^2(\Gamma)} g_2(s)\;ds =0 \text{ in }  L^2 \left( 0, T\right).
    \end{equation}
    Using the change of variables $s \mapsto t-s$ in \eqref{expectation} and differentiating with respect to time $t$, we obtain
    \begin{flalign*}
        \langle\partial_{\nu} v_2(t), \eta\rangle_{L^2(\Gamma)}&=-\frac{1}{g_2(0)} \int_0^t \langle\partial_{\nu} v_2(s), \eta\rangle_{L^2(\Gamma)}\frac{\partial}{\partial t}g_2(t-s)\;ds.
    \end{flalign*}
    Taking the absolute value on both sides yields
    \[
    \left| \langle \partial_\nu v_2(t), \eta \rangle_{L^2(\Gamma)} \right|
    \le \frac{1}{|g_2(0)|} \int_0^t \left| \langle \partial_\nu v_2(s), \eta \rangle_{L^2(\Gamma)} \right| \left| \frac{\partial}{\partial t} g_2(t-s) \right| ds.
    \]
    Since $g_2 \in C^1([0,T])$, applying Gronwall's inequality, we obtain $\langle \partial_\nu v_2(t), \eta \rangle_{L^2(\Gamma)} = 0$. Since $\eta \in C^\infty_c(\Gamma)$ is arbitrary and $C^\infty_c(\Gamma)$ is dense in $L^2(\Gamma)$, we derive $\partial_\nu v_2(x,t) = 0$ in $L^2((0,T); L^2(\Gamma))$.Thus, $v_2$ satisfies the MGT equation \eqref{eq_determinsitic_MGT_for_control} with $\phi = f_2 \in L^2(D)$. By Theorem \ref{thm_UCP}, we conclude that $f_2(x) = 0$ a.e. in $D$.

    Second, to prove $f_{1}(x) = 0$ a.e. in $D$, we divide the proof into three cases based on the value of the Hurst parameter H. \\
    \noindent \textbf{Case 1: $H = \frac{1}{2}$.} According to Itô's isometry \eqref{eq:H=1/2}, taking the variance yields
    \begin{equation}\label{eq:var=1/2}
        \operatorname{Var}\left[ \langle\partial_{\nu} u(t),\eta\rangle_{L^2(\Gamma)} \right]=\int_0^t (\langle\partial_{\nu} v_1(t-s),\eta\rangle_{L^2(\Gamma)})^2 g^2_1(s)\;ds =0 \text{ in }  L^2 \left(0, T\right).
    \end{equation}
    Using the same change of variables and differentiating with respect to $t$, we obtain
    \begin{align*}
	   \langle\partial_{\nu} v_1(t),\eta\rangle_{L^2(\Gamma)} ^2 &= - \frac{1}{g_1^2(0)} \int_0^t \langle\partial_{\nu} v_1(s),\eta\rangle_{L^2(\Gamma)} ^2 \frac{\partial}{\partial t} g_1^2(t-s)\;ds \\
	   &\le \frac{1}{g_1^2(0)} \int_0^t \langle\partial_{\nu} v_1(s),\eta\rangle_{L^2(\Gamma)} ^2 \left| \frac{\partial}{\partial t} g_1^2(t-s) \right| \;ds.
    \end{align*}
    Applying Gronwall's inequality as above, we obtain $\langle \partial_\nu v_1(t), \eta \rangle_{L^2(\Gamma)} = 0.$ Consequently, $\partial_\nu v_1(x,t) = 0$ in $L^2((0,T); L^2(\Gamma))$. Invoking Theorem \ref{thm_UCP} again, we conclude that $f_1(x) = 0$ a.e. in $D$. \\ 
    \noindent \textbf{Case 2: $H > \frac{1}{2}$.} Applying \eqref{eq:H>1/2}, taking the variance yields
    \begin{flalign}\label{eq:var>1/2}
	   &\operatorname{Var}\left[ \langle\partial_{\nu}u(t),\eta\rangle_{L^2(\Gamma)} \right] \nonumber\\
       = &\int_0^t \int_0^t \langle\partial_\nu v_1( t-s),\eta\rangle_{L^2(\Gamma)} g_1(s) \langle\partial_\nu v_1( t-r),\eta\rangle_{L^2(\Gamma)} \phi(r-s) \, dr ds=0\text{ in }  L^2 \left(0, T\right).
    \end{flalign}
    For a fixed $t$, let $\langle\partial_\nu v_1( t-s),\eta\rangle_{L^2(\Gamma)}g_1(s) = \psi_t(s)$. Recalling $\phi(r-s) = H(2H-1)|r-s|^{2H-2}$, equation \eqref{eq:var>1/2} becomes
    \begin{equation}
	   \int_0^t \int_0^t \psi_t(s) \psi_t(r) |r-s|^{2H-2} \, ds dr = 0.	\label{eq_variance}
    \end{equation}
    According to \cite[chapter 2.1]{biagini2008stochastic}, we have the identity:
    \begin{flalign} 
	   |r-s|^{2H-2} = \frac{(rs)^{H-\frac{1}{2}}}{\beta(2-2H, H-\frac{1}{2})} \int_0^{r \wedge s} v^{1-2H} (r-v)^{H-\frac{3}{2}} (s-v)^{H-\frac{3}{2}} \, dv	\label{eq_kernel_for_phi}.
    \end{flalign}
    Substituting \eqref{eq_kernel_for_phi} into \eqref{eq_variance} yields
    \begin{flalign*}
        &\int_0^t \int_0^t \int_0^{r \wedge s} \psi_t(s) \psi_t(r) v^{1-2H} (sr)^{H-\frac{1}{2}} (r-v)^{H-\frac{3}{2}} (s-v)^{H-\frac{3}{2}} \, dv ds dr \nonumber\\
        &= \int_0^t \left( \int_v^t \psi_t(s) s^{H-\frac{1}{2}} (s-v)^{H-\frac{3}{2}} \, ds \right)^2 v^{1-2H} \, dv = 0.
    \end{flalign*}
    This implies that for almost every $v \in (0,t)$:
    \begin{flalign}
       \int_{v}^{t}\psi_t(s)s^{H-\frac{1}{2}}(s-v)^{H-\frac{3}{2}}ds = 0.\label{eq_temp}
    \end{flalign}
    Defining $\tilde{\psi}_t(s) =: \psi_t(s) s^{H-\frac{1}{2}}$, equation \eqref{eq_temp} can be rewritten using the Riemann-Liouville fractional integral \cite[Definition 2.1]{jin2021fractional} as
    \begin{flalign}\label{eq:R-L.integral}
        ({}_{v}I_{t}^{H-\frac{1}{2}}\tilde{\psi}_t)(v) = 0 \quad \text{for a.e. } v \in (0,t),
    \end{flalign}
    where $\left({}_v I_t^{H-\frac{1}{2}}h\right)(v)=\frac{1}{\Gamma(H-\frac{1}{2})}\int_v^t (s-v)^{H-\frac{1}{2}-1}h(s)\; ds$. Note that for $H>\frac{1}{2}$, $\int_{0}^{T}|\tilde{\psi}_t(s)|ds = \int_{0}^{T}|\psi_t(s)|s^{H-\frac{1}{2}}ds \le C (\int_{0}^{T}|\psi_t(s)|^{2}ds)^{\frac{1}{2}} < \infty$, which ensures $\tilde{\psi}_t \in L^{1}(0,t)$. Using the semigroup property of fractional integrals ${}_{a}I_{x}^{\alpha}{}_{a}I_{x}^{\beta}f = {}_{a}I_{x}^{\alpha+\beta}f$ (see \cite[Theorem 2.1]{jin2021fractional}), and noting $H+1 > 0$, we apply ${}_{v}I_{t}^{\frac{3}{2}-H}$ to both sides of \eqref{eq:R-L.integral} to obtain:
    \begin{equation*}
        {}_{v}I_{t}^{\frac{3}{2}-H}{}_{v}I_{t}^{H-\frac{1}{2}}\tilde{\psi}_t(v) = {}_{v}I_{t}^{1}\tilde{\psi}_t(v) = \int_{v}^{t}\tilde{\psi}_t(s)ds = 0 \quad \text{for a.e. } v \in (0,t).
    \end{equation*}
    Differentiating with respect to $v$ yields:
    \begin{equation*}
        \tilde{\psi}_t(v) = -\partial_{v}\left(\int_{v}^{t}\tilde{\psi}_t(s)ds\right) = 0 \quad \text{for a.e. } v \in (0,t).
    \end{equation*}
    Consequently, $\psi_t(s) = 0$ for almost all $s \in (0,t)$. This means $\langle\partial_{\nu}v_{1}(t-s),\eta\rangle_{L^{2}(\Gamma)}g_{1}(s) = 0$ almost everywhere. Integrating this over $(0,t)$ we get
    \begin{equation}
        \int_{0}^{t}\langle\partial_{\nu}v_{1}(t-s),\eta\rangle_{L^{2}(\Gamma)}g_{1}(s)ds = 0.\label{eq:v1-integral}
    \end{equation}
    Applying the exact same arguments to \eqref{eq:v1-integral} as those used for \eqref{expectation}, we deduce $\partial_\nu v_1(x,t) = 0$, which by Theorem \ref{thm_UCP} immediately implies $f_1(x) = 0$ a.e. in $D$.\\
    \noindent \textbf{Case 3: $H < \frac{1}{2}$.} Applying \eqref{eq:H<1/2}, taking the variance yields
    \begin{equation*}
    \operatorname{Var}[\langle\partial_\nu u(t), \eta\rangle_{L^2(\Gamma)}] = \int_0^t \left( K_H^*(t, \cdot) \left[ \langle\partial_\nu v_1( t-\cdot),\eta\rangle_{L^2(\Gamma)} g_1(\cdot) \right] \right)^2 (s) \, ds = 0.
    \end{equation*}
    According to \cite[Section 5.1.2]{nualart2006malliavin}, the operator $K_H^*(t)$ associated with the fractional Wiener integral satisfies
    \begin{equation*}
    \left[ K_H^*(t) \varphi(\cdot) \right] (s) = d_H s^{\frac{1}{2}-H} D_{t-}^{\frac{1}{2}-H} \left( u^{H-\frac{1}{2}} \varphi(u) \right)(s).
    \end{equation*}
    where $d_H = c_H \Gamma\left(H + \frac{1}{2}\right)$ and
    \begin{flalign*}
        D_{t-}^{\frac{1}{2}-H}\varphi(s)=\frac{1}{\Gamma(H+1/2)}\left(\frac{\varphi(s)}{\left(t-s\right)^{1/2-H}}+\left(\frac{1}{2}-H\right)\int_s^t\frac{\varphi(s)-\varphi(\tau)}{\left(\tau-s\right)^{3/2-H}}\; d\tau\right).
    \end{flalign*} 
    Consequently, for a.e. $t \in (0,T)$ and $s \in (0,t)$,
    \begin{equation}
        d_H s^{\frac{1}{2}-H}D_{t-}^{\frac{1}{2}-H}\left( \langle\partial_{\nu}v_1(t-u),\eta\rangle_{L^2(\Gamma)}g_1(u)u^{H-\frac{1}{2}} \right)(s) = 0.
    \end{equation}
    Since $d_H s^{\frac{1}{2}-H} \ne 0$ for $s \in (0,t)$, it follows that
    \begin{equation*}
        D_{t-}^{\frac{1}{2}-H}\left( \langle\partial_{\nu}v_1(t-u),\eta\rangle_{L^2(\Gamma)}g_1(u)u^{H-\frac{1}{2}} \right)(s) = 0 \quad \text{for a.e. } s \in (0,t).
    \end{equation*}
    For a fixed $t$, let $\varphi_t(u) := \langle\partial_{\nu}v_1(t-u),\eta\rangle_{L^2(\Gamma)}g_1(u)u^{H-\frac{1}{2}}$. Since $\partial_{\nu}v_1(x,0) = 0$, we have $\varphi_t(t) = \langle\partial_{\nu}v_1(0),\eta\rangle_{L^2(\Gamma)}g_1(t)t^{H-\frac{1}{2}} = 0$. Given that $\partial_{\nu}v_1 \in H^1((0,T); L^2(\partial D))$, $g_1 \in C^1([0,T])$ and the mapping $u \mapsto u^{H-\frac{1}{2}}$ is smooth on $[\epsilon, t]$ for any arbitrary $\epsilon > 0$, we conclude that $\varphi_t$ is absolutely continuous on the interval $([\epsilon, t])$. We can express $\varphi_t(u) = -({}_{u}I_{t-}^1 \varphi_t') (u)$ on the interval $[\epsilon, t]$. Applying the fractional integral operator ${}_{u}I_{t-}^{\frac{1}{2}-H}$ and utilizing the semigroup property of fractional integrals, we derive on $[\epsilon, t]$:
    \begin{flalign*}
        {}_{u}I_{t-}^{\frac{1}{2}-H} D_{t-}^{\frac{1}{2}-H}\varphi_t &= -{}_{u}I_{t-}^{\frac{1}{2}-H} D_{t-}^{\frac{1}{2}-H} {}_{u}I_{t-}^1 \varphi_t' = -{}_{u}I_{t-}^{\frac{1}{2}-H} D_{t-}^{\frac{1}{2}-H} {}_{u}I_{t-}^{\frac{1}{2}-H} {}_{u}I_{t-}^{H+\frac{1}{2}} \varphi_t' \\
        &= -{}_{u}I_{t-}^{\frac{1}{2}-H} {}_{u}I_{t-}^{H+\frac{1}{2}} \varphi_t' = -{}_{u}I_{t-}^1 \varphi_t' = \varphi_t.
    \end{flalign*}
    Therefore, $D_{t-}^{\frac{1}{2}-H}\varphi_t = 0$ yields $\varphi_t(s) = 0$ on $[\epsilon, t]$. Since $\epsilon > 0$ is arbitrary, $\varphi_t = 0$ a.e. on $(0,t)$, which means
    \begin{equation*}
        \langle\partial_{\nu}v_1(t-s),\eta\rangle_{L^2(\Gamma)}g_1(s)s^{H-\frac{1}{2}} = 0 \quad \text{for a.e. } s \in (0,t).
    \end{equation*}
    Since $s^{H-\frac{1}{2}} \ne 0$, we obtain $\langle\partial_{\nu}v_1(t-s),\eta\rangle_{L^2(\Gamma)}g_1(s) = 0$ for a.e. $t \in (0,T)$ and $s \in (0,t)$. Following the exact same arguments from this point onward as in Case 2, we derive
     \begin{equation}
        \int_{0}^{t}\langle\partial_{\nu}v_{1}(t-s),\eta\rangle_{L^{2}(\Gamma)}g_{1}(s)ds = 0.\label{eq:v1-integral_2}
    \end{equation}
    and conclude $f_1(x) = 0$ a.e. in $D$.
\end{proof}
\subsection{Uniqueness of recovering \texorpdfstring{$g_i(x)$}{gi(x)}}
Now we are aiming to prove that if the spatial functions $f_i(x)$ are already known, we can recover $g_i(t)$ uniquely on $[0, T]$ with some extra conditions of $f_i$ for $H \in \left(0, 1\right)$. The proof is mainly based on Titchmarsh convolution theorem which characterizes the null spaces of Volterra convolution operator. It plays an important role in proving the uniqueness of recovering $g_2(t)$ and $g_1(t)$.
\begin{theorem}[Titchmarsh convolution theorem]\cite[Theorem VII]{https://doi.org/10.1112/plms/s2-25.1.283} \label{thm_Titchmarsh}
    If $\phi(t)$ and $\psi(t)$ are integrable functions, such that
    \begin{equation*}
        \int_0^t \phi(s) \psi(t - s) \, ds = 0
    \end{equation*}
    holds almost everywhere in the interval $0 < t < T$, then $\phi(t) = 0$ almost everywhere in $(0, \lambda)$, and $\psi(t) = 0$ almost everywhere in $(0, \mu)$, where $\lambda + \mu \ge T$. 
\end{theorem} 
We now turn to the proof of Theorem \ref{thm_unique-g}.
\begin{proof}
    Let $h_i(t) := \langle\partial_\nu v_i(t), \eta\rangle_{L^2(\Gamma)}$ for an arbitrary test function $\eta \in C^\infty_c(\Gamma)$. As established in the proof of Theorem \ref{thm_unique-f}, the condition $\partial_\nu u(x,t) = 0$ a.s. implies the following two convolution equations
    \begin{align}
        \int_0^t h_2(t-s) g_2(s) ds = 0, \label{eq:con-h2}\\
        \int_0^t h_1(t-s) g_1(s) ds = 0.\label{eq:con-h1}
    \end{align}
    Equation \eqref{eq:con-h2} is derived directly from the expectation \eqref{expectation} and \eqref{eq:con-h1} follows from the variance equations \eqref{eq:v1-integral} for $H>\frac{1}{2}$, and \eqref{eq:v1-integral_2} for $H<\frac{1}{2}$.
    For the critical case $H=\frac{1}{2}$, by the variance \eqref{eq:var=1/2}, since the integrand is non-negative, we deduce that $\langle\partial_\nu v_1(t-s), \eta\rangle_{L^2(\Gamma)} g_1(s) = 0$ a.e. in $(0,t)$, which consequently also yields \eqref{eq:con-h1}.
    
    Since $f_i \in H^2(D)\cap H_0^1(D)$, the hidden regularity of MGT equation solutions ensures $h_i \in H^3(0,T)$. Thus $h_i \in C^2(0, T)$. Recalling the initial conditions of the equation \eqref{MGT_2} that $v_i$ satisfies, we have $h_i(0) = h_i'(0) = 0$, and the second derivative is evaluated as
    \begin{align*}
        \frac{d^2}{dt^2}h_i(0) = \langle\partial_\nu \partial_t^2 v_i(0), \eta\rangle_{L^2(\Gamma)} = \langle\partial_\nu f_i, \eta\rangle_{L^2(\Gamma)}.
    \end{align*}
    Since $\partial_\nu f_i \not\equiv 0$ in $L^2(\Gamma)$ for $i=1,2$, and $C^\infty_c(\Gamma)$ is dense in $L^2(\Gamma)$, we can choose $\eta \in C^\infty_c(\Gamma)$ such that $\frac{d^2}{dt^2}h_i(0) \neq 0$. This implies $h_i(t) \not\equiv 0$ on any interval $(0, \lambda)$. Applying Theorem \ref{thm_Titchmarsh} to \eqref{eq:con-h2} and \eqref{eq:con-h1}, we immediately deduce $g_1(t) = g_2(t) = 0$ a.e. in $(0,T)$. 
\end{proof}
\begin{remark}
    We would like to point out that the assumption $\partial_{\nu}f_i \ne 0$ on $\Gamma$ in $L^2(\Gamma)$ in Theorem \ref{thm_unique-g} is physically and geometrically reasonable. Since $f_i \in H_0^1(D)$, the trace of the spatial source on the boundary is zero ($f_i|_{\Gamma} = 0$). Consequently, the tangential derivative of $f_i$ vanishes on $\Gamma$, implying that the gradient $\nabla f_i$ is completely aligned with the normal vector, i.e., $\nabla f_i(x) = \partial_{\nu}f_i(x)\nu(x)$ for $x \in \Gamma$. Therefore, the assumption $\partial_{\nu}f_i \not\equiv 0$ simply means that the gradient of the spatial source profile does not vanish everywhere on the observation boundary $\Gamma$. From a practical point of view, this ensures that the spatial distribution of the source actively ``pushes'' against the observation boundary, yielding a non-trivial signal to reconstruct the temporal component.
\end{remark}

\section{Acknowledgments}
The authors would like to express their sincere gratitude to Professor Peijun Li and Xv Wang for carefully reading an earlier version of the manuscript and providing constructive comments that improved its presentation.


\bibliographystyle{plain} 
\bibliography{references}  

@article{https://doi.org/10.1002/mma.1576,
author = {Marchand, R. and McDevitt, T. and Triggiani, R.},
title = {An abstract semigroup approach to the third-order {Moore-Gibson-Thompson} partial differential equation arising in high-intensity ultrasound: structural decomposition, spectral analysis, exponential stability},
journal = {Mathematical Methods in the Applied Sciences},
volume = {35},
number = {15},
pages = {1896-1929},
doi = {https://doi.org/10.1002/mma.1576},
url = {https://onlinelibrary.wiley.com/doi/abs/10.1002/mma.1576},
eprint = {https://onlinelibrary.wiley.com/doi/pdf/10.1002/mma.1576},
year = {2012}
}

@book{pazy2012semigroups,
  title={Semigroups of linear operators and applications to partial differential equations},
  author={Pazy, Amnon},
  year={2012},
  publisher={Springer Science \& Business Media}
}

@book{pub.1039910508,
 author={Thom{\'e}e, Vidar},
 date = {2006},
 doi = {10.1007/3-540-33122-0},
 title = {Galerkin Finite Element Methods for Parabolic Problems},
 url = {https://app.dimensions.ai/details/publication/pub.1039910508},
 year = {2006},
 publisher={Springer Science \& Business Media}
}

@article{article,
author = {Kaltenbacher, Barbara and Lasiecka, Irena and Marchand, R.},
year = {2011},
month = {01},
pages = {971-988},
title = {Wellposedness and exponential decay rates for the {Moore-Gibson-Thompson} equation arising in high intensity ultrasound},
volume = {40},
journal = {Control and Cybernetics}
}

@article{L2011GlobalUF,
    title={Global Uniqueness for an Inverse Stochastic Hyperbolic Problem with Three Unknowns},
    author={Qi L{\"u} and Xu Zhang},
    journal = {Communications on Pure and Applied Mathematics},
    volume = {68},
    number = {6},
    pages = {948-963},
    doi = {https://doi.org/10.1002/cpa.21503},
    url = {https://onlinelibrary.wiley.com/doi/abs/10.1002/cpa.21503},
    eprint = {https://onlinelibrary.wiley.com/doi/pdf/10.1002/cpa.21503},
    year = {2015}
}

@article{Bucci2020,
  author  = {Bucci, Francesca and Pandolfi, Luciano},
  title   = {On the regularity of solutions to the {Moore-Gibson-Thompson} equation: a perspective via wave equations with memory},
  journal = {Journal of Evolution Equations},
  year    = {2020},
  volume  = {20},
  number  = {3},
  pages   = {837--867},
  doi     = {10.1007/s00028-019-00549-x},
  url     = {https://doi.org/10.1007/s00028-019-00549-x}
}

@misc{article_Arancibia_2022,
url = {https://doi.org/10.1515/jiip-2020-0090},
title = {An inverse problem for {Moore-Gibson-Thompson} equation arising in high intensity ultrasound},
author = {Rogelio Arancibia and Rodrigo Lecaros and Alberto Mercado and Sebastián Zamorano},
pages = {659--675},
volume = {30},
number = {5},
journal = {Journal of Inverse and Ill-posed Problems},
doi = {10.1515/jiip-2020-0090},
year = {2022},
lastchecked = {2026-08-02}
}

@article{https://doi.org/10.1112/plms/s2-25.1.283,
    author = {CARTWRIGHT, MARY L.},
    title = {THE ZEROS OF CERTAIN INTEGRAL FUNCTIONS},
    journal = {The Quarterly Journal of Mathematics},
    volume = {os-1},
    number = {1},
    pages = {38-59},
    year = {1930},
    month = {01},
    issn = {0033-5606},
    doi = {10.1093/qmath/os-1.1.38},
    url = {https://doi.org/10.1093/qmath/os-1.1.38},
    eprint = {https://academic.oup.com/qjmath/article-pdf/os-1/1/38/4482898/os-1-1-38.pdf}
}

@article{Lassas_2023,
doi = {10.1088/1361-6420/acdab9},
url = {https://doi.org/10.1088/1361-6420/acdab9},
year = {2023},
month = {jun},
publisher = {IOP Publishing},
volume = {39},
number = {8},
pages = {084001},
author = {Lassas, Matti and Li, Zhiyuan and Zhang, Zhidong},
title = {Well-posedness of the stochastic time-fractional diffusion and wave equations and inverse random source problems},
journal = {Inverse Problems}
}

@article{fractionalBrownian,
	author = {Duncan, T. and Maslowski, B. and Pasik-Duncan, Bozenna},
	year = {2011},
	month = {11},
	pages = {},
	title = {Fractional Brownian motion and stochastic equations in Hilbert spaces},
	volume = {02},
	journal = {Stochastics and Dynamics},
	doi = {10.1142/S0219493702000340}
}

@book{biagini2008stochastic,
	title     = {Stochastic Calculus for Fractional Brownian Motion and Applications},
	author    = {Biagini, Francesca and Hu, Yaozhong and Øksendal, Bernt and Zhang, Tusheng},
	year      = {2008},
	publisher = {Springer London},
	address   = {London},
	isbn      = {978-1-84628-797-8},
	doi       = {10.1007/978-1-84628-797-8}
}

@book{jin2021fractional,
	title     = {Fractional Differential Equations},
	author    = {Jin, Bangti},
	year      = {2021},
	publisher = {Springer International Publishing},
	address   = {Cham},
	isbn      = {978-3-030-76043-4},
	doi       = {10.1007/978-3-030-76043-4}
}

@book{nualart2006malliavin,
title = {The Malliavin Calculus and Related Topics},
  author = {Nualart, David},
  year = {2006},
  publisher = {Springer Berlin Heidelberg},
  edition = {2},
  series = {Probability and Its Applications},
  doi = {10.1007/3-540-28329-3},
  isbn = {978-3-540-28328-7}
}

@book{mishura2008stochastic,
  title     = {Stochastic Calculus for Fractional Brownian Motion and Related Processes},
  author    = {Mishura, Yuliya S.},
  year      = {2008},
  publisher = {Springer Berlin Heidelberg},
  doi       = {10.1007/978-3-540-75873-0},
  isbn      = {978-3-54-075872-3}
}

@article{articleExponentialDecay,
author = {Kaltenbacher, Barbara and Lasiecka, Irena},
year = {2012},
month = {01},
pages = {},
title = {Exponential decay for low and higher energies in the third order linear {Moore-Gibson-Thompson} equation with variable viscosity},
volume = {1},
journal = {Palestine Journal of Mathematics [electronic only]}
}

@article{doi:10.2514/8.8418,
author = {MOORE, FRANKLIN K. and GIBSON, WALTER E.},
title = {Propagation of Weak Disturbances in a Gas Subject to Relaxation Effects},
journal = {Journal of the Aerospace Sciences},
volume = {27},
number = {2},
pages = {117-127},
year = {1960},
doi = {10.2514/8.8418},

URL = { 
    
        https://doi.org/10.2514/8.8418
    
    

},
eprint = { 
    
        https://doi.org/10.2514/8.8418
    
    

}

}

@article{thompson1972compressible,
  title={Compressible-fluid dynamics},
  author={Thompson, Philip A and Beavers, Gordon S},
  journal={Journal of Applied Mechanics},
  volume={39},
  number={2},
  pages={},
  year={1972},
  publisher={ASME International},
  url={https://api.semanticscholar.org/CorpusID:262446187}
}

@article{XiaoliFeng2022,
title = {An inverse source problem for the stochastic wave equation},
journal = {Inverse Problems and Imaging},
volume = {16},
number = {2},
pages = {397-415},
year = {2022},
issn = {1930-8337},
doi = {10.3934/ipi.2021055},
url = {https://www.aimsciences.org/article/id/f6b8d6bc-274e-488a-9fab-caeb2b596561},
author = {Xiaoli Feng and Meixia Zhao and Peijun Li and Xu Wang}
}

@article{https://doi.org/10.1002/mma.70574,
author = {Chang, Kuijian and Gong, Yuxuan and Xu, Xiang},
title = {Recovering Initial Values and a Random Source Simultaneously for a Damped Wave Equation},
journal = {Mathematical Methods in the Applied Sciences},
volume = {49},
number = {9},
pages = {9948-9966},
doi = {https://doi.org/10.1002/mma.70574},
url = {https://onlinelibrary.wiley.com/doi/abs/10.1002/mma.70574},
eprint = {https://onlinelibrary.wiley.com/doi/pdf/10.1002/mma.70574},
year = {2026}
}

@article{Feng_2020,
doi = {10.1088/1361-6420/ab6503},
url = {https://doi.org/10.1088/1361-6420/ab6503},
year = {2020},
month = {feb},
publisher = {IOP Publishing},
volume = {36},
number = {4},
pages = {045008},
author = {Feng, Xiaoli and Li, Peijun and Wang, Xu},
title = {An inverse random source problem for the time fractional diffusion equation driven by a fractional Brownian motion},
journal = {Inverse Problems}
}

@article{Tindel2003,
  author = {Tindel, S. and Tudor, C. A. and Viens, F.},
  title = {Stochastic evolution equations with fractional {Brownian} motion},
  journal = {Probability Theory and Related Fields},
  year = {2003},
  volume = {127},
  number = {2},
  pages = {186--204},
  doi = {10.1007/s00440-003-0282-2},
  url = {https://doi.org/10.1007/s00440-003-0282-2},
  issn = {1432-2064}
}

@article{doi:10.1137/21M1429138,
author = {Li, Peijun and Wang, Xu},
title = {An Inverse Random Source Problem for the Biharmonic Wave Equation},
journal = {SIAM/ASA Journal on Uncertainty Quantification},
volume = {10},
number = {3},
pages = {949-974},
year = {2022},
doi = {10.1137/21M1429138},

URL = { 
    
        https://doi.org/10.1137/21M1429138
    
    

},
eprint = { 
    
        https://doi.org/10.1137/21M1429138
    
    

}
}

@article{Fu_2024,
doi = {10.1088/1361-6420/ad49cd},
url = {https://doi.org/10.1088/1361-6420/ad49cd},
year = {2024},
month = {may},
publisher = {IOP Publishing},
volume = {40},
number = {7},
pages = {075001},
author = {Fu, Song-Ren and Yao, Peng-Fei and Yu, Yongyi},
title = {Inverse problem of recovering a time-dependent nonlinearity appearing in third-order nonlinear acoustic equations},
journal = {Inverse Problems}
}

@book{evans2022partial,
  title={Partial differential equations},
  author={Evans, Lawrence C},
  volume={19},
  year={2022},
  publisher={American Mathematical Society}
}

@article{Decreusefond1999,
	author   = {Decreusefond, L. and {\"U}st{\"u}nel, A. S.},
	title    = {Stochastic Analysis of the Fractional Brownian Motion},
	journal  = {Potential Analysis},
	year     = {1999},
	volume   = {10},
	number   = {2},
	pages    = {177--214},
	doi      = {10.1023/A:1008634027843},
	url      = {https://doi.org/10.1023/A:1008634027843},
	issn     = {1572-929X}
}

@book{hille1996functional,
  title={Functional analysis and semi-groups},
  author={Hille, Einar and Phillips, Ralph Saul},
  volume={31},
  year={1996},
  publisher={American Mathematical Society}
}

\end{document}